\documentclass{amsart}
\usepackage{amssymb}
\usepackage{hyperref}
\usepackage{tikz-cd}
\usepackage{mathtools}
\usepackage{mathrsfs}
\author{Pramod N. Achar}
\address{Department of Mathematics\\
  Louisiana State University\\
  Baton Rouge, LA 70803\\
  U.S.A.}
\email{pramod.achar@math.lsu.edu}

\author{Tamanna Chatterjee}
\address{Department of Mathematics\\
  University of Notre Dame\\
  255 Hurley Bldg.\\
  Notre Dame, IN 46556\\
  U.S.A}
\email{tchatter@nd.edu}

\thanks{The first author was supported by NSF Grant No.~DMS-2202012.}
\subjclass{14M15; 32S60}

\title{Equivariant sheaves for classical groups acting on Grassmannians}

\newcommand{\Z}{\mathbb{Z}}

\newcommand{\C}{\mathbb{C}}
\newcommand{\Q}{\mathbb{Q}}

\newcommand{\bk}{\Bbbk}

\newcommand{\ubk}{\underline{\bk}}

\newcommand{\Gm}{\mathbb{G}_{\mathrm{m}}}

\newcommand{\cL}{\mathcal{L}}

\newcommand{\Hc}{\mathsf{H}_{\mathrm{c}}}
\newcommand{\cB}{\mathcal{B}}
\newcommand{\cF}{\mathcal{F}}

\newcommand{\cN}{\mathcal{N}}
\newcommand{\Lie}{\mathrm{Lie}}

\newcommand{\mc}{\mathcal}
\newcommand{\mb}{\mathbb}

\newcommand{\lag}{\langle}
\newcommand{\rag}{\rangle}

\newcommand{\eal}{\end{align*}}

\newcommand{\bal}{\begin{align*}}

\newcommand{\GL}{\mathrm{GL}}

\newcommand{\Sp}{\mathrm{Sp}}
\newcommand{\SO}{\mathrm{SO}}

\newcommand{\rO}{\mathrm{O}}
\newcommand{\Iso}{\mathrm{Iso}}
\newcommand{\Isoc}{\mathrm{Iso}^\circ}
\newcommand{\GIsoc}{\mathrm{GIso}^\circ}
\newcommand{\pr}{\mathrm{pr}}

\newcommand{\iso}{{\mathrm{iso}}}

\newcommand{\cO}{\mathcal{O}}
\newcommand{\cE}{\mathcal{E}}
\newcommand{\cC}{\mathcal{C}}
\newcommand{\Gr}{\mathrm{Gr}}

\newcommand{\tC}{\widetilde{\cC}}
\newcommand{\hC}{\widehat{\cC}}
\newcommand{\eC}{E}

\newcommand{\uk}{\underline{k}}
\newcommand{\ur}{\underline{r}}
\newcommand{\hX}{\widehat{X}}
\newcommand{\hF}{\widehat{F}}
\newcommand{\hmu}{\widehat{\mu}}
\newcommand{\D}{\mathbb{D}}
\newcommand{\cH}{\mathcal{H}}
\newcommand{\fgl}{\mathfrak{gl}}
\newcommand{\fp}{\mathfrak{p}}

\numberwithin{equation}{section}
\newtheorem{thm}{Theorem}[section]
\newtheorem{lem}[thm]{Lemma}
\newtheorem{prop}[thm]{Proposition}
\newtheorem{cor}[thm]{Corollary}

\theoremstyle{definition}

\theoremstyle{remark}
\newtheorem{rmk}[thm]{Remark}

\DeclareMathOperator{\ind}{ind}

 \DeclareMathOperator{\Stab}{Stab}

\DeclareMathOperator{\rhm}{R\mathscr{H}\text{\kern -3pt{\calligra\large om}}\, }

\DeclareMathOperator{\Hom}{Hom}
\DeclareMathOperator{\End}{End}

\DeclareMathOperator{\spn}{span}

\newcommand{\tra}{\dag}

\newcommand{\coh}{\mathsf{H}}
\newcommand{\Db}{D^{\mathrm{b}}}

\DeclareMathOperator{\rad}{rad}
\newcommand{\sgn}{\mathrm{sgn}}

\newcommand{\simto}{\overset{\sim}{\to}}
\newcommand{\id}{\mathrm{id}}

\begin{document}

\begin{abstract}
Let $V$ be a finite-dimensional complex vector space.  Assume that $V$ is a direct sum of subspaces each of which is equipped with a nondegenerate symmetric or skew-symmetric bilinear form.  In this paper, we introduce a stratification of the Grassmannian $\Gr_k(V)$ related to the action of the appropriate product of orthogonal and symplectic groups, and we study the topology of this stratification.  The main results involve sheaves with coefficients in a field of characteristic other than $2$.  We prove that there are ``enough'' parity sheaves, and that the hypercohomology of each parity sheaf also satisfies a parity-vanishing property.

This situation arises in the following context: let $x$ be a nilpotent element in the Lie algebra of either $G = \Sp_N(\C)$ or $G = \SO_N(\C)$, and let $V = \ker x \subset \C^N$.  Our stratification of $\Gr_k(V)$ is preserved by the centralizer $G^x$, and we expect our results to have applications in Springer theory for classical groups.
\end{abstract}

\maketitle

%%%%%%%%%%%%%%%%%%%%%%%%%%%%%%%%%%%%%%%%%%%%%%%%%%%%%%%%%%%%%%%%%%%%%%%%%%%
\section{Introduction}
%%%%%%%%%%%%%%%%%%%%%%%%%%%%%%%%%%%%%%%%%%%%%%%%%%%%%%%%%%%%%%%%%%%%%%%%%%%

%--------------------------------------------------------------------------
\subsection{Overview and statement of the main results}
\label{ss:overview}
%--------------------------------------------------------------------------

Parity sheaves, introduced by Juteau--Mautner--Williamson in~\cite{jmw:ps}, have become a powerful tool in modular geometric representation theory.  One especially important case is that of flag varieties and their variants (such as affine or partial flag varieties, including Grassmannians) equipped with the Schubert stratification.  In this setting, parity sheaves give rise to the notion of the \emph{$p$-canonical basis} of the Hecke algebra, and ultimately to a number of major advances of the past decade, including~\cite{amrw:kdkmg, ar:psag, rw:tmpcb, rw:stt}.

The goal of this paper is to study parity sheaves on Grassmannians equipped with various stratifications \emph{other} than the Schubert stratification.  These stratifications, defined in terms of the action of classical groups, are motivated by an anticipated future application to Springer theory.  This application will be discussed at the end of the introduction.

Let us introduce the notation needed to formulate the main results.  Let $B$ be a finite-dimensional complex vector space equipped with a decomposition
\[
B = B_1 \oplus \cdots \oplus B_m.
\]
(On a first reading, the reader may wish to assume that $m = 1$.)  Suppose that each $B_i$ is equipped with a nondegenerate bilinear form $\lag{-},{-}\rag_{B_i}$ that is either symmetric or skew-symmetric.  Let $\Iso(B_i) \subset \GL(B_i)$ be the subgroup that preserves $\lag{-},{-}\rag_{B_i}$, and let $\Isoc(B_i) \subset \Iso(B_i)$ be its identity component.  Thus
\[
\Isoc(B_i) = 
\begin{cases}
\SO(B_i) & \text{if $\lag{-},{-}\rag_{B_i}$ is symmetric,} \\
\Sp(B_i) & \text{if $\lag{-},{-}\rag_{B_i}$ is skew-symmetric.}
\end{cases}
\]
Let
\[
I_B = \Isoc(B_1) \times \cdots \times \Isoc(B_m) \subset \GL(B).
\]
In Section~\ref{sec:dsums}, we will introduce larger group $Q_B \subset \GL(B)$ that has the form
\[
Q_B = I_B \ltimes (\text{a normal connected unipotent group}).
\]
(If $m = 1$, the unipotent part is trivial, and we have $Q_B = I_B = \Isoc(B_1)$.)

Let $k$ be an integer with $0 \le k \le \dim B$, and consider the Grassmannian $\Gr_k(B)$ of $k$-dimensional subspaces of $V$.  The groups $Q_B$ and $I_B$ act on $\Gr_k(B)$.  We can consider the equivariant derived category
\[
\Db_{Q_B}(\Gr_k(B),\bk)
\]
of sheaves of $\bk$-modules, where $\bk$ is a field (more general rings will be allowed in the body of the paper).

The main results of the paper are summarized in items~\eqref{it:classify}--\eqref{it:alter-exist} below:
\begin{enumerate}
\item We classify the $Q_B$-orbits on $\Gr_k(B)$.\label{it:classify}
\item We determine the $Q_B$-equivariant fundamental group of each $\Gr_k(B)$-orbit.\label{it:classify-locsys}
\end{enumerate}
It turns out that each such fundamental group is a product of copies of $\Z/2\Z$.  As a consequence, if the characteristic of $\bk$ is not $2$, then the category of $Q_B$-equivariant $\bk$-local systems on each orbit is semisimple.
\begin{enumerate}
\setcounter{enumi}{2}
\item For each $Q_B$-orbit $\cO \subset \Gr_k(B)$ and each irreducible $Q_B$-equivariant $\bk$-local system $\cL$ on $\cO$, we construct an indecomposable parity sheaf\label{it:par-exist}
\[
\cE(\cO,\cL)
\]
that is supported on $\overline{\cO}$ and satisfies $\cE(\cO,\cL)|_\cO \cong \cL[\dim \cO]$.
\item We prove that the hypercohomology of parity sheaves satisfies:\label{it:par-cohom}
\[
\coh^i(\Gr_k(\cB), \cE(\cO,\cL)) = 0
\qquad\text{if $i \not\equiv \dim \cO \pmod 2$.}
\]
\end{enumerate}

For item~\eqref{it:par-exist}, once we have constructed $\cE(\cO,\cL)$, the general theory from~\cite{jmw:ps} tells us that it is unique up to isomorphism.  The construction requires some auxiliary results, including the following:
\begin{enumerate}\setcounter{enumi}{4}
\item For each $Q_B$-orbit $\cO \subset \Gr_k(B)$, we construct a resolution of singularities $\mu: X \to \overline{\cO}$ whose fibers have cohomology only in even degrees.\label{it:resoln-exist}
\end{enumerate}
In the terminology of~\cite{jmw:ps}, these are called ``even resolutions.''  These resolutions are already enough to prove the existence of parity sheaves $\cE(\cO,\ubk)$ associated to \emph{constant} local systems.  But when $\cO$ admits nontrivial local systems, more is needed.
\begin{enumerate}\setcounter{enumi}{5}
\item For each $Q_B$-orbit $\cO \subset \Gr_k(B)$, we construct a proper map $\hmu: \hX \to \overline{\cO}$ with the following properties: (i)~$\hX$ is smooth; (ii)~the fibers of $\hmu$ have cohomology only in even degrees; (iii)~$\hmu^{-1}(\cO)$ is simply connected; (iv)~$\hmu: \hmu^{-1}(\cO) \to \cO$ is a smooth morphism.\label{it:alter-exist}
\end{enumerate}
Parts (iii) and (iv) of this statement are the key to constructing $\cE(\cO,\cL)$ for nontrivial $\cL$.  One could ask whether part~(iv) could be improved to say that $\hmu: \hmu^{-1}(\cO) \to \cO$ is \'etale (so that $\hmu$ is an alteration).  We do not know whether such a statement is true.

%--------------------------------------------------------------------------
\subsection{Motivation}
%--------------------------------------------------------------------------

The present paper is motivated by an anticipated application to \emph{Mautner's cleanness conjecture}.  Before explaining this connection, we review some background on Springer theory.

In Lusztig's work~\cite{lus:icc} on the generalized Springer correspondence (with $\overline{\Q}_\ell$-coefficients), one key result states that \emph{cuspidal} simple perverse sheaves are \emph{clean}, i.e., their stalks vanish outside of one orbit.  The cleanness phenomenon plays a crucial role in the computation of stalks of perverse sheaves on the nilpotent cone~\cite{lus:cs5}, and in the theory of character sheaves.

About ten years ago, Mautner conjectured that these cuspidal perverse sheaves remain clean when they are reduced modulo $p$, with some possible exceptions for very small $p$.  Mautner's conjecture has been verified in type $A$ and the exceptional types, and for a few small-rank examples in types $B$, $C$, and $D$.  In general, it is currently open for classical groups outside of type $A$.  

Here is a sketch of the kind of computation that goes into checking Mautner's conjecture by hand.  Suppose $G = \Sp_{2N}$, and let $P$ be a parabolic subgroup with Levi factor $L = \GL_k \times \Sp_{2N-2k}$.  Suppose $\cF$ is a perverse sheaf on the nilpotent cone of $\Sp_{2N-2k}$, and identify the latter with a subset of the nilpotent cone of $L$.  Then we can consider the parabolic induction $\ind_{L \subset P}^G \cF$, a perverse sheaf on the nilpotent cone $\cN_G$ of $G$.  We might wish to compute the stalk
\[
(\ind_{L \subset P}^G \cF)|_x
\qquad\text{at a point $x \in \cN_G$.}
\]

To do this, we need some additional input.  Let $G^x \subset \Sp_{2N}$ be the stabilizer of $x$ under the adjoint action, and let $G^{x,\circ}$ be its identity component.  An explicit description of the reductive part of $G^x$ or $G^{x,\circ}$ can be found in the classic textbook~\cite{cm:nosla}, and in many other sources as well: we have
\[
G^{x,\circ} = I_B \ltimes (\text{a normal connected unipotent group})
\qquad
\text{where $B = \ker x$.}
\]
Here, the decomposition $B = B_1 \oplus \cdots \oplus B_m$ is obtained from the Jordan blocks of $m$, and one can show that 
\[
(\ind_{L \subset P}^G \cF)|_x \cong R\Gamma(\text{a certain $Q_B$-equivariant complex on $\Gr_k(B)$.})
\]
The results of the present paper can give us substantial control over both sides of this equation, and we hope to use them to prove Mautner's cleanness conjecture in a future paper.

%--------------------------------------------------------------------------
\subsection{Organization of the paper}
%--------------------------------------------------------------------------

In Section~\ref{sec:affpave}, we prove a few preparatory lemmas about the topology of ``isotropic Grassmannians'' with respect to possibly degenerate bilinear forms.

Sections~\ref{sec:orbits} and~\ref{sec:resoln} are devoted the $m = 1$ case.  That is, we consider a vector space $V$ equipped with a nondegenerate symmetric or skew-symmetric bilinear form, and we study the orbits of $\SO(V)$ or $\Sp(V)$, respectively, on $\Gr_k(V)$.  In these sections, we classify orbits and equivariant local systems; we exhibit resolutions of singularities for orbit closures; and we exhibit ``universal covering submersions'' for non-simply-connected orbits.

In Section~\ref{sec:dsums}, we lay the groundwork for the case of general $m$.  We define the groups $I_B$ and $Q_B$, and we classify orbits and local systems, completing items~\eqref{it:classify} and~\eqref{it:classify-locsys} from Section~\ref{ss:overview}.

In Section~\ref{sec:resoln2}, we exhibit a resolution of singularities for each orbit closure (item~\eqref{it:resoln-exist} above).  In Section~\ref{sec:locsys2}, we prove item~\eqref{it:alter-exist} from Section~\ref{ss:overview}.  Finally, the main results on parity sheaves (items~\eqref{it:par-exist} and~\eqref{it:par-cohom} above) are proved in Section~\ref{sec:parity}.

%--------------------------------------------------------------------------
\subsection{General conventions}
\label{ss:conv}
%--------------------------------------------------------------------------

Throughout the paper, we treat Grassmannians and other complex algebraic varieties as equipped with the analytic topology, and we consider sheaves and (singular) cohomology with respect to this topology.  

We will typically denote by $\bk$ the ring of coefficients for singular or equivariant cohomology, or for sheaves.  In the early sections of the paper, most statements allow $\bk$ to be an arbitrary commutative ring; sometimes, we also require $2$ to be invertible.  In Section~\ref{sec:parity}, we will impose the following assumption
\begin{equation}\label{eqn:jmw}
\begin{minipage}{4in}
The ring $\bk$ is either a field of characteristic${}\ne 2$, or a complete local prinicipal ideal domain in which $2$ is invertible.
\end{minipage}
\end{equation}

Regardless of the assumptions on $\bk$, we will use the phrase
\[
\text{``$\coh^\bullet(X,\bk)$ is even and free''}
\]
to mean
\[
\coh^i(X,\bk) =
\begin{cases}
\text{a free $\bk$-module of finite rank} & \text{if $i$ is even,} \\
0 & \text{if $i$ is odd.}
\end{cases}
\]
A ``$\bk$-local system of finite type'' means a locally constant sheaf of $\bk$-modules whose stalks are finitely generated $\bk$-modules.  A $\bk$-local system is said to be ``locally free'' if its stalks are free $\bk$-modules.

%%%%%%%%%%%%%%%%%%%%%%%%%%%%%%%%%%%%%%%%%%%%%%%%%%%%%%%%%%%%%%%%%%%%%%%%%%%
\section{Preliminaries on isotropic Grassmannians}
\label{sec:affpave}
%%%%%%%%%%%%%%%%%%%%%%%%%%%%%%%%%%%%%%%%%%%%%%%%%%%%%%%%%%%%%%%%%%%%%%%%%%%

Let $X$ be a variety.  Recall that an \emph{affine paving} of $X$ is a decomposition
\[
X = \bigsqcup_{\alpha \in I} S_\alpha
\]
into finitely many locally closed subvarieties such that the following conditions hold:
\begin{enumerate}
\item The indexing set $I$ can be equipped with a total order $I = (\alpha_1 < \cdots < \alpha_N)$ such that for $1 \le i \le N$, the subset
\[
S_{\alpha_1} \cup S_{\alpha_2} \cup \cdots \cup S_{\alpha_i}
\]
is closed.
\item Each $S_\alpha$ is isomorphic to an affine space $\mathbb{A}^{k_\alpha}$.
\end{enumerate}
If $X$ admits an affine paving, then $\Hc^\bullet(X,\bk)$ is even and free.

Let $V$ be a vector space equipped with a (possibly degenerate) symmetric or skew-symmetric bilinear form $\lag{-},{-}\rag$.  For any subspace $H \subset V$, we set
\begin{align*}
H^\perp &= \{ v \in V \mid \text{$\lag v, w \rag = 0$ for all $w \in H$} \} \\
&= \{ v \in V \mid \text{$\lag w, v \rag = 0$ for all $w \in H$} \}.
\end{align*}
We will usually call this the \emph{orthogonal complement} of $H$, although that is probably a misnomer when the form is degenerate.  A subspace $H \subset V$ is called \emph{isotropic} if $H \subset H^\perp$.  For an integer $k$ with $0 \le k \le \dim V$, the \emph{isotropic Grassmannian} is the variety
\[
\Gr_k(V)^\iso = \{ H \in \Gr_k(V) \mid H \subset H^\perp \}.
\]
The \emph{radical} of $V$ is
\[
\rad V = V^\perp.
\]

\begin{lem}\label{lem:gr-affpave}
Let $V$ be a finite-dimensional complex vector space equipped with a (possibly degenerate) symmetric or skew-symmetric bilinear form $\lag{-},{-}\rag$.  Let $M_1 \subset M_2 \subset \cdots \subset M_c$ be a sequence of isotropic subspaces of $V$.  For $0 \le k \le \dim V$, the variety $\Gr_k(V)^\iso$ admits an affine paving
\[
\Gr_k(V)^\iso = \bigsqcup_{\alpha \in I} S_\alpha
\]
such that for each piece $S_\alpha$ and each $i$ with $1 \le i \le c$, the function
\[
S_\alpha \to \Z
\qquad\text{given by}\qquad
H \mapsto \dim H \cap M_i
\]
is constant.
\end{lem}
This lemma is certainly well known in the cases where the bilinear form $\lag{-},{-}\rag$ is either identically zero, or else nondegenerate: in these cases, it follows from the Bruhat decomposition.  

\begin{proof}
We proceed by induction on $\dim V$.  If $\dim V = 0$, then $k = 0$, and $\Gr_k(V)^\iso$ is a singleton.  The statement is obvious in this case.

Assume now that $\dim V > 0$.  If $\lag{-},{-}\rag$ is degenerate, choose a $1$-dimensional space $L \subset \rad V$.  Otherwise, let $L$ be any $1$-dimensional subspace of $M_1$.  In both cases, $L$ is automatically isotropic.  Choose a complement $W$ to $L$ in $L^\perp$, so that $L^\perp = L \oplus W$.  Let $p_1: L^\perp \to L$ and $p_2: L^\perp \to W$ be the projection maps.  If $H \subset L^\perp$, then 
\begin{multline*}
\dim p_2(H) \cap p_2(M_i) = \dim p_2^{-1}(p_2(H) \cap p_2(M_i)) -1 = \dim (H+L) \cap (M_i + L) -1 \\
= \dim (H+L) + \dim (M_i+L) - \dim (H+M_i+L) - 1.
\end{multline*}
Comparing this to $\dim H \cap M_i = \dim H + \dim M_i - \dim (H+M_i)$, we find that
\begin{equation}\label{eqn:hm-dim}
\dim H \cap M_i =
\begin{cases}
\dim p_2(H) \cap p_2(M_i) + 1& \text{if $L \subset H$ and $L \subset M_i$,} \\
\dim p_2(H) \cap p_2(M_i) & \text{if $L \not\subset H$ and $L \subset M_i$,} \\
\dim p_2(H) \cap p_2(M_i) & \text{if $L \subset H$ and $L \not\subset M_i$,} \\
\dim p_2(H) \cap p_2(M_i) -1 & \text{if $L \not\subset H$, $L \not \subset M_i$, and $L \subset H + M_i$,} \\
\dim p_2(H) \cap p_2(M_i) & \text{if $L \not\subset H + M_i$.}
\end{cases}
\end{equation}

Now let
\begin{align*}
X_1 &= \{H \in \Gr_k(V)^\iso \mid L \subset H \}, \\
X_2 &= \{H \in \Gr_k(V)^\iso \mid L \not\subset H,\ H \subset L^\perp\}, \\
X_3 &= \{H \in \Gr_k(V)^\iso \mid H \not\subset L^\perp\}.
\end{align*}
It is enough to show that each of these locally closed subsets separately admits a suitable affine paving.

For $X_1$, observe that an isotropic subspace $H$ that contains $L$ is necessarily contained in $L^\perp$.  The map $p_2$ induces an isomorphism
\[
X_1 \cong \Gr_{k-1}(W)^\iso.
\]
By induction, $\Gr_{k-1}(W)^\iso$ admits an affine paving such that the function $F \mapsto \dim F \cap p_2(M_i)$ is constant on each stratum.  We transfer this affine paving to $X_1$. Since $\dim p_2(H)\cap p_2(M_i)$ is constant on each stratum and $L$ and $L$ is fixed, so the condition in ~\eqref{eqn:hm-dim} is also constant. Therefore the function $H \mapsto \dim H \cap M_i$ is constant on each stratum in $X_1$.

Next, we turn to $X_2$.  Consider the variety
\[
E = \{ (F, f) \mid F \in \Gr_k(W)^\iso,\ f \in \Hom(F,L) \}.
\]
This is a vector bundle over $\Gr_k(W)^\iso$.  For $H \in X_2$, the map $p_2$ induces an isomorphism $H \cong p_2(H)$.  Let $f_H: p_2(H) \to L$ be the composition
\[
f_H: p_2(H) \xrightarrow[\sim]{\text{induced by $p_2$}} H \xrightarrow{p_1} L.
\]
Then we have a map
\[
\theta: X_2 \to E
\qquad\text{given by}\qquad H \mapsto (p_2(H), f_H).
\]
In fact, this is an isomorphism: the inverse map sends $(F,f)$ to $H = \{(f(v),v) \in L \oplus W \mid v \in F \}$.  Thus, $X_2$ is a vector bundle over $\Gr_k(W)^\iso$.  By induction, $\Gr_k(W)^\iso$ admits an affine paving such that $F \mapsto \dim F \cap p_2(M_i)$ is constant along strata.  Taking preimages along $X_2 \to \Gr_k(W)^\iso$, we obtain an affine paving of $X_2$.  By~\eqref{eqn:hm-dim}, the function $H \mapsto \dim H \cap M_i$ is constant on each stratum in $X_2$.

Finally, we consider $X_3$.  If $X_3$ is nonempty, then $L^\perp \ne V$, i.e., $L$ is not contained in the radical of $\lag{-},{-}\rag$.  Our set-up implies that this can happen only when $\lag{-},{-}\rag$ is nondegenerate.  Assume this from now.  We also have $L \subset M_1$, and since all the $M_i$ are isotropic, we have
\begin{equation}\label{eqn:lmc-perp}
L \subset M_1 \subset \cdots \subset M_c \subset L^\perp.
\end{equation}
In addition, the restriction of $\lag{-},{-}\rag$ to $W$ is nondegenerate, and $W^\perp$ is $2$-dimen\-sion\-al and contains $L$.  Choose an isotropic complement $R$ to $L$ in $W^\perp$, and then choose some nonzero vector $r_0 \in R$.

Let
\[
X_2' = \{ F \in \Gr_{k-1}(V)^\iso \mid L \not\subset F,\ F \subset L^\perp \}.
\]
This variety is analogous to $X_2$, and the same reasoning shows that $X_2'$ has an affine paving such that $F \mapsto \dim F \cap M_i$ is constant on strata. Let $u \in L^\perp/F$, we define a linear map $\lag u, {-} \rag: F \to \C$.  this map is well defined as $\lag u,v\rag=0$ for all $u, v\in F$, which is because $F\subset F^\perp$. The assignment $u \mapsto \lag u,-\rag$ is a linear map $L^\perp/F \to F^*$.   Consider the element $\lag -r_0, {-}\rag \in F^*$.  Its preimage in $L^\perp/F$, given concretely by
\[
\{ u \in L^\perp/F \mid \text{$\lag u,v \rag = \lag -r_0,v \rag$ for all $v \in F$} \}
\]
is an affine subspace of $L^\perp/F$.  We deduce that the variety
\[
E' = \{ (F,u) \mid \text{$F \in X_2'$, $u \in L^\perp/F$, and $\lag u,v \rag = \lag -r_0,v \rag$ for all $v \in F$} \}
\]
is an affine bundle over $X_2'$.  The affine paving of $X_2'$ induces an affine paving of $E'$.

For $H \in X_3$, we have $H \cap L^\perp \in X_2'$.  Moreover, the composition $H \to H/(H \cap L^\perp) \cong (H+L^\perp)/L^\perp$ is surjective. As $H$ is one dimensional and not contained in $L^\perp$, hence $H+L^\perp\cong V$. This implies $H/(H \cap L^\perp) \cong V/L^\perp$. Now we claim   $V/L^\perp\cong R$. The intersection  $R\cap L^\perp$ must be $0$. If not then $R\subset L^\perp$. But  $L^\perp \cap W^\perp$ is one dimensional and hence $L=L^\perp \cap W^\perp$. So $R$ can not be contained in both $W^\perp$ and $L^\perp$ and we are done with our claim. Therefore $R\hookrightarrow V/L^\perp$ and by dimension equality $V/L^\perp \cong R$. Finally we get a surjective map $H\to R$. so there is a vector $\tilde r_H \in H$ that maps to $r_0$. We have $\tilde r_H - r_0 \in L^\perp$.  Since $H$ is isotropic, for any $v \in H \cap L^\perp$, we have $\lag \tilde r_H, v \rag = 0$, so
\[
\lag \tilde r_H - r_0, v \rag = \lag -r_0, v\rag.
\]
We can thus consider the map
\[
\theta': X_3 \to E'
\qquad\text{given by}\qquad
\theta'(H) = (H \cap L^\perp, \tilde r_H - r_0 + (H \cap L^\perp)).
\]

The remainder of the argument is slightly different depending on whether $\lag{-},{-}\rag$ is symmetric or skew-symmetric.  In the skew-symmetric case, $\theta'$ is an isomorphism: the inverse map sends a pair $(F,u) \in E'$ to the space $H = F + \C(u + r_0)$.  Thus, $X_3$ has an affine paving induced by that $X_2'$.  In view of~\eqref{eqn:lmc-perp}, we have
\[
\dim H \cap M_i = \dim (H \cap L^\perp) \cap M_i,
\]
and thus the function $H \mapsto \dim H \cap M_i$ is constant on strata.

It remains to consider the symmetric case.  We begin by describing $E'$ another way.  For $F \in X_2'$, the projection maps $p_1: L^\perp \to L$ and $p_2: L^\perp \to W$ induce an isomorphism
\[
L^\perp/F \cong L \oplus W/p_2(F).
\]
For $u \in L^\perp/F$, write $u = u_1 + u_2$ with $u_1 \in L$ and $u_2 \in W/p_2(F)$.  Then $\lag u, v\rag = \lag u_2, v\rag$ for all $v \in F$, so we can identify
\[
E' = \left\{ (F,u_1,u_2) \,\Big|\,
\begin{array}{c}
\text{$F \in X_2'$, $u_1 \in L$, $u_2 \in W/p_2(F)$,} \\
\text{and $\lag u_2,v \rag = \lag -r_0,v \rag$ for all $v \in F$}
\end{array}
\right\}.
\]
In this description, there is no condition on $u_1$.  Consider the subset
\[
E'' = \{ (F, u_1, u_2) \in E' \mid \lag u_1, r_0\rag = \textstyle -\frac{1}{2}\lag u_2 + r_0, u_2 + r_0\rag \}.
\]
We need to check that the expression ``$\lag u_2+r_0, u_2+r_0\rag$'' makes sense, because $W/p_2(F) \oplus R$ does \emph{not} have an induced bilinear form.  Since we are starting with a point in $E'$, for $v \in F$, we have
\[
\lag u_2 + v +r_0, u_2 +v+r_0\rag = \lag u_2 + r_0, u_2+r_0\rag + 2\lag u_2+r_0,v\rag + \lag v, v\rag.
\]
On the right-hand side, the middle term vanishes by the definition of $E'$, and the last term vanishes because $F$ is isotropic.  We conclude that $E''$ is well-defined.  Its definition says simple that the $u_1$ component is determined by $u_2$.  Thus, $E''$ is isomorphic to the variety
\[
\{ (F, u_2) \in E' \mid \text{$F \in X_2'$, $u_2 \in W/p_2(F)$, and $\lag u_2,v \rag = \lag -r_0,v \rag$ for all $v \in F$} \}.
\]
In particular, $E''$ is also an affine space bundle over $X_2'$.

Note that
\[
\lag u_1 + u_2 + r_0, u_1 + u_2 + r_0\rag = 2\lag u_1, r_0\rag + \lag u_2 + r_0, u_r + r_0\rag.
\]
In view of this, the definition of $E''$ can be rewritten as
\[
E'' = \{ (F,u) \in E' \mid \lag u + r_0, u+ r_0\rag = 0 \}.
\]
In the symmetric case, it is clear that the image of $\theta'$ is contained in $E''$, since $\tilde r_H$ is an isotropic vector.  In fact, $\theta'$ induces an isomorphism $X_3 \cong E''$: the inverse map again sends $(F,u) \in E'$ to $H = F + \C(u + r_0)$.  We conclude that $X_3$ has an affine paving with the desired properties.
\end{proof}

\begin{cor}\label{cor:gr-oddvan}
For $V$ as in Lemma~\ref{lem:gr-affpave} and for any commutative ring $\bk$, we have that $\coh^\bullet(\Gr_k(V)^\iso, \bk)$ is even and free.
\end{cor}

The next lemma will be a useful technical tool for inductively constructing affine pavings later in the paper.

\begin{lem}\label{lem:gr-pave}
Let $V$ be a finite-dimensional complex vector space equipped with a (possibly degenerate) symmetric or skew-symmetric bilinear form $\lag{-},{-}\rag$.  Let $M_1 \subset M_2 \subset \cdots \subset M_c$ be a sequence of isotropic subspaces of $V$.  Let $0 \le r \le k \le \dim V$, and consider the variety
\[
Y(r,k) = \{ (R,H) \in \Gr_r(M_1 \cap \rad V) \times \Gr_k(V)^\iso \mid R \subset H \}.
\]
Then $Y(r,k)$ admits a partition into locally closed subvarieties
\[
Y(r,k) = \bigsqcup_{\alpha \in I} \tilde S_\alpha
\]
with the following two properties:
\begin{enumerate}
\item For each $\alpha$ and each $i$ with $1 \le i \le c$, the function $\tilde S_\alpha \to \Z$ given by $(R,H) \mapsto \dim H \cap M_i$ is constant.
\item For each $\alpha$, the projection map $\tilde S_\alpha \to \Gr_r(M_1 \cap \rad V)$ is a locally trivial fibration whose fibers are affine spaces.
\end{enumerate}
\end{lem}

\begin{proof}
	We can define $\tilde S_\alpha=\{(R,H)\in \Gr_r(M_1\cap \rad V)\times S_\alpha \mid R\subset H\}$. From Lemma \ref{lem:gr-affpave},  (1) is obvious.
	For (2), the map, $\Gr_r(M_1\cap \rad V)\times S_\alpha \to \Gr_r(M_1\cap \rad V)$ is a locally trivial fibration with fibre isomorphic to $S_\alpha$, which is affine. The space, $\tilde S_\alpha$ is a closed subspace of  
	$\Gr_r(M_1\cap \rad V)\times S_\alpha $. Therefore the proof follows.
\end{proof}

We conclude this section with the following lemma on the nondegenerate case.

\begin{lem}\label{lem:griso-conn}
Let $V$ be a finite-dimensional complex vector space equipped with a nondegenerate symmetric or skew-symmetric bilinear form $\lag{-},{-}\rag$.  Let $0 \le k \le \frac{1}{2}\dim V$.
\begin{enumerate}
\item The variety $\Gr_k(V)^\iso$ is smooth, and each connected component is simply connected.
\item If $\lag{-},{-}\rag$ is symmetric and $\dim V = 2k$, then $\Gr_k(V)^\iso$ has exactly two connected components.  In all other cases, $\Gr_k(V)^\iso$ is connected.
\end{enumerate}
\end{lem}
The assertions in Lemma~\ref{lem:griso-conn} are well known.  Alternatively, they follow from Propositions~\ref{prop:orbit-symp} and~\ref{prop:orbit-orth} below: in the notation of those propositions, we have $\Gr_k(V)^\iso = \cC_{k,0}(V)$.

%%%%%%%%%%%%%%%%%%%%%%%%%%%%%%%%%%%%%%%%%%%%%%%%%%%%%%%%%%%%%%%%%%%%%%%%%%%
\section{Symplectic and orthogonal group orbits on Grassmannians}
\label{sec:orbits}
%%%%%%%%%%%%%%%%%%%%%%%%%%%%%%%%%%%%%%%%%%%%%%%%%%%%%%%%%%%%%%%%%%%%%%%%%%%

In this section, we work in the following setting: let $V$ be a complex vector space of dimension $n < \infty$, and assume that it is equipped with a nondegenerate bilinear form
\[
\lag{-},{-}\rag: V \times V \to \C,
\]
assumed to be either symmetric or skew-symmetric.  Then one may consider the isotropy group $\Iso(V)$ and its identity component $\Isoc(V)$.

The goal of this section is to classify the orbits of $\Isoc(V)$ (i.e., either $\SO(V)$ or $\Sp(V)$) on the Grassmannian $\Gr_k(V)$.  This classification is certainly known, but we are unaware of a reference that includes enough details about stabilizers and equivariant cohomology for our needs.  For this reason, we give complete proofs here.

%--------------------------------------------------------------------------
\subsection{Stratification by rank}
\label{ss:orbit-gen}
%--------------------------------------------------------------------------

We begin with some preliminary lemmas that apply in both the symmetric and skew-symmetric cases.

\begin{lem}\label{lem:v4-decomp}
Let $H \subset V$ be a subspace of dimension $k$, and let $r = k - \dim H \cap H^\perp$.  There exist subspaces $M_1, M_2, M_3, M_4 \subset V$ with the following properties:
\begin{gather}
V = M_1 \oplus M_2 \oplus M_3 \oplus M_4, \label{eqn:v4sum}\\
\dim M_1 = k- r, \quad
\dim M_2 = r, \quad
\dim M_3 = n-2k+r, \quad
\dim M_4 = k - r, \notag\\
H \cap H^\perp = M_1, 
\qquad H = M_1 \oplus M_2, \qquad H^\perp = M_1 \oplus M_3. \notag
\end{gather}
Moreover, the restriction of $\lag{-},{-}\rag$ to the summands in~\eqref{eqn:v4sum} satisfies
\begin{multline}\label{eqn:v4sum-pair}
\lag{-},{-}\rag: M_i \times M_j \to \C 
\qquad\text{is}\qquad \\
\begin{cases}
\text{perfect} & \text{if $\{i,j\} = \{1,4\}$, or if $i = j = 2$, or if $i = j = 3$,} \\
0 & \text{otherwise.}
\end{cases}
\end{multline}
\end{lem}
\begin{proof}
Set $M_1 = H \cap H^\perp$, and choose $M_2$ to be any complement to $M_1$ in $H$.  Choose $M_3$ to be any complement to $M_1$ in $H^\perp$.  Finally, choose $M_4$ to be any complement to $H+ H^\perp = M_1 \oplus M_2 \oplus M_3$.  It is straightforward to check that the restriction of $\lag{-},{-}\rag$ to these summands has the properties claimed in~\eqref{eqn:v4sum-pair}.
\end{proof}

\begin{lem}\label{lem:v4-orbit}
Let $0 \le k \le n$, and let 
\[
\cC_{k,r}(V) = \{ H \in \Gr_k(V) \mid \dim H - \dim H \cap H^\perp = r \}.
\]
If $\cC_{k,r}(V)$ is nonempty, then it is a single $\Iso(V)$-orbit.  Moreover, it is a single $\Isoc(V)$-orbit if one of the following holds:
\begin{itemize}
\item The form $\lag{-},{-}\rag$ is skew-symmetric.
\item The form $\lag{-},{-}\rag$ is symmetric, and either $r > 0$ or $n - 2k +r > 0$.
\end{itemize}
\end{lem}
\begin{proof}
Suppose $H$ and $H'$ are two points in $\cC_{k,r}(V)$.  Invoke Lemma~\ref{lem:v4-decomp} twice to obtain two decompositions
\begin{align*}
V &= M_1 \oplus M_2 \oplus M_3 \oplus M_4, \\
V &= M_1' \oplus M_2' \oplus M_3' \oplus M_4'.
\end{align*}
Choose any linear isomorphism $g_1: M_1 \simto M'_1$.  Thanks to the pairing properties in~\eqref{eqn:v4sum-pair}, we have an adjoint map $g_1^\tra: M_4' \simto M_4$.  Next, since $M_2$ and $M_2'$ are vector spaces of the same dimension, both equipped with nondegenerate bilinear forms of the same type (either symmetric or skew-symmetric), there exists a pairing-preserving isomorphism $g_2: M_2 \simto M_2'$.  We similarly get a pairing-preserving isomorphism $g_3: M_3 \simto M_3'$.  Let
\[
g= g_1 \oplus g_2 \oplus g_3 \oplus (g_1^\tra)^{-1}: V \to V.
\]
This is a linear automorphism of $V$ that preserves the bilinear form, i.e., an element of $\Iso(V)$.  By construction, $g(H) = H'$, so $H$ and $H'$ are in the same $\Iso(V)$-orbit.

If $\lag{-},{-}\rag$ is skew-symmetric, then $\Iso(V) = \Isoc(V) = \Sp(V)$, so $H$ and $H'$ are in the same $\Isoc(V)$-orbit.

Suppose now that $\lag{-},{-}\rag$ is symmetric, and that $r > 0$.  If $\det g = 1$, then $g \in \SO(V)$, and we already have that $H$ and $H'$ are in the same $\Isoc(V)$-orbit.  Now suppose that $\det g = -1$.  Since $\dim M_2 = r > 0$, there exists an element $s \in \rO(M_2) \smallsetminus \SO(M_2)$, i.e., an automorphism of $M_2$ with determinant $-1$.  Let
\[
g' = g \circ (\id_{M_1} \oplus s \oplus \id_{M_3} \oplus \id_{M_4}).
\]
Then $g'$ also sends $H$ to $H'$, but it has determinant $1$, so $H$ and $H'$ are in the same orbit for $\Isoc(V) = \SO(V)$.

Finally,  if $\lag{-},{-}\rag$ is symmetric and $n - 2k + r > 0$, similar reasoning applies using an element of $\rO(M_3) \smallsetminus \SO(M_3)$.
\end{proof}

Suppose we have chosen a decomposition of $V$ as in~\eqref{eqn:v4sum}.  Given a linear map $\theta: V \to V$, we can write it as a $4 \times 4$ matrix with respect to this decomposition. Let $\theta^\tra: V \to V$ be the adjoint map to $\theta$, i.e., the map such that $\lag \theta({-}),{-}\rag = \lag {-},\theta^\tra({-})\rag$.  Thanks to~\eqref{eqn:v4sum-pair}, the adjoint can be computed ``entrywise,'' as follows:
\begin{equation}\label{eqn:v4-adjoint}
\text{if}\quad
\theta = 
\left[\begin{smallmatrix}
a & b & c & d \\
e & f & g & h \\
i & j & k & l \\
m & n & p & q
\end{smallmatrix}\right]
\qquad\text{then}\qquad
\theta^\tra = 
\left[\begin{smallmatrix}
q^\tra & h^\tra & l^\tra & d^\tra \\
n^\tra & f^\tra & j^\tra & b^\tra \\
p^\tra & g^\tra & k^\tra & c^\tra \\
m^\tra & e^\tra & i^\tra & a^\tra
\end{smallmatrix}\right].
\end{equation}

\begin{lem}\label{lem:v4-stab}
Let $H \subset V$ be a subspace of dimension $k$, and choose a decomposition of $V$ as in Lemma~\ref{lem:v4-decomp}.  Then the stabilizer of $H$ in $\Iso(V)$ is given by
\[
\Stab_{\Iso(V)}(H) \cong (\GL(M_1) \times \Iso(M_2) \times \Iso(M_3)) \ltimes A,
\]
where $A \subset \GL(V)$ is the unipotent group given by
\[
A = 
\left\{
\left[\begin{smallmatrix}
1 & b & c & d \\
& 1 & & -b^\tra  \\
& & 1 & -c^\tra  \\
& & & 1
\end{smallmatrix}
\right]
\, \Bigg|\, d + d^\tra = -bb^\tra - cc^\tra \right\}.
\]
\end{lem}
\begin{proof}
Let $H = M_1 \oplus M_2$.  Then
\[
H^\perp = M_1 \oplus M_3
\qquad\text{and}\qquad H \cap H^\perp = M_1.
\]
An element of $\Iso(V)$ that preserves $H$ must also preserve $H^\perp$ and $H \cap H^\perp$.  If $\theta \in \GL(V)$ preserves $H$, $H^\perp$, and $H \cap H^\perp$, then it has the form
\[
\theta = \left[\begin{smallmatrix}
a & b & c & d \\
 & e & & f \\
 & & g & h \\
 &&& k
 \end{smallmatrix}\right]
\]
where $a$, $e$, $g$, and $k$ are invertible.   Then $\theta^\tra$ and $\theta^{-1}$ are given by
\[
\left[\begin{smallmatrix}
k^\tra & f^\tra & h^\tra & d^\tra \\
 & e^\tra & & b^\tra \\
 & & g^\tra & c^\tra \\
 &&& a^\tra
\end{smallmatrix}\right]
\quad\text{and}\quad
\left[\begin{smallmatrix}
a^{-1} & -a^{-1}be^{-1} & -a^{-1}cg^{-1} & a^{-1}be^{-1}fk^{-1} + a^{-1}cg^{-1}hk^{-1} - a^{-1}dk^{-1}\\
 & e^{-1} & & -e^{-1}f k^{-1} \\
 & & g^{-1} & -g^{-1}hk^{-1} \\
 &&& k^{-1}
 \end{smallmatrix}\right]
\]
respectively.  For $\theta$ to lie in $\Iso(V)$, we must have $\theta^\tra = \theta^{-1}$.  We deduce that
\[
\Stab_{\Iso(V)}(H) = 
\left\{
\left[\begin{smallmatrix}
a & b & c & d \\
& e & & -eb^\tra (a^\tra)^{-1} \\
& & g & -gc^\tra (a^\tra)^{-1} \\
& & & (a^\tra)^{-1}
\end{smallmatrix}
\right]
\, \Bigg|\,
\begin{array}{@{}c@{}}
\text{$a \in \GL(M_1)$, $e \in \Iso(M_2)$, $g \in \Iso(M_3)$,} \\
\text{and $d^\tra = - a^{-1} b b^\tra - a^{-1} c c^\tra - a^{-1} d a^\tra$}
\end{array}
\right\}
\]
This can be rewritten as
\[\left\{\left[ \begin{smallmatrix}
	a &&&\\ &e&&\\ && g& \\&&& (a^{\tra})^{-1}  
\end{smallmatrix}\right]
\left[\begin{smallmatrix}
1 & a^{-1}b & a^{-1}c & a^{-1}d \\
& 1 & & -b^\tra (a^\tra)^{-1} \\
& & 1 & -c^\tra (a^\tra)^{-1} \\
& & & 1
\end{smallmatrix}
\right]
\, \Bigg|\,
\begin{array}{@{}c@{}}
\text{$a \in \GL(M_1)$, $e \in \Iso(M_2)$, $g \in \Iso(M_3)$,} \\
\text{and $d^\tra = - a^{-1} b b^\tra - a^{-1} c c^\tra - a^{-1} d a^\tra$}
\end{array}
\right\},
\]
and after an appropriate change of coordinates, this matches the statement of the lemma.
\end{proof}

\begin{rmk}\label{rmk:giso-defn}
Later in the paper, we will occasionally need to consider the group
\[
\GIsoc(V) = \C^\times \times \Isoc(V).
\]
Let $\GIsoc(V)$ act on $V$ by having the $\C^\times$ factor act by scaling.  Of course, in the induced action on $\Gr_k(V)$, the $\C^\times$-factor acts trivially.  It is easy to see that:
\begin{enumerate}
\item The $\GIsoc(V)$-orbits on $\Gr_k(V)$ coincide with the $\Isoc(V)$-orbits.
\item For any $H \in \Gr_k(V)$, the natural map
\[
\Stab_{\Isoc(V)}(H)/\Stab_{\Isoc(V)}(H)^\circ \to \Stab_{\GIsoc(V)}(H)/\Stab_{\GIsoc(V)}(H)^\circ
\]
is a bijection.  As a consequence, any $\Isoc(V)$-equivariant local system on an $\Isoc(V)$-orbit $\cC$ is also $\GIsoc(V)$-equivariant.
\end{enumerate}
\end{rmk}

%--------------------------------------------------------------------------
\subsection{Tangent spaces}
\label{ss:m1-tangent}
%--------------------------------------------------------------------------

In this subsection, we use Lemma~\ref{lem:v4-decomp} to carry out some computations in the tangent space to a point of $\Gr_k(V)$.

\begin{lem}\label{lem:gr-tan}
Let $V$ be a finite-dimensional vector space, and let $H \in \Gr_k(V)$.  Let $T_H\Gr_k(V)$ be the tangent space to $\Gr_k(V)$ at $H$.
\begin{enumerate}
\item There is a canonical isomorphism\label{it:gr-tan}
\[
T_H\Gr_k(V) \cong \Hom(H,V/H).
\]
\item There is a (noncanonical) open immersion\label{it:gr-open}
\[
\Hom(H,V/H) \hookrightarrow \Gr_k(V)
\]
that sends $0 \in \Hom(H,V/H)$ to $H \in \Gr_k(V)$, and that induces the isomorphism of tangent spaces from part~\eqref{it:gr-tan}.
\end{enumerate}
\end{lem}
\begin{proof}
\eqref{it:gr-tan}~The group $\GL(V)$ acts transitively on $\Gr_k(V)$.  Let $P$ be the stabilizer of $H$ in $\GL(V)$; this is a parabolic subgroup.  Then $T_H\Gr_k(V)$ is identified with the quotient of $\fgl(V)$ by the subspace $\fp = \Lie(P)$.  Choose a complement $H'$ to $H$ in $V$.  We of course have identifications
\[
\fgl(V) = \End(V) = \Hom(H \oplus H',V) = \Hom(H,V) \oplus \Hom(H',V).
\]
The subspace $\fp \subset \fgl(V)$ is the subspace of $\End(V)$ consisting of maps that preserve $H$: thus,
\[
\fp = \Hom(H,H) \oplus \Hom(H',V).
\]
The cokernel of $\fp \hookrightarrow \fgl(V)$ is thus identified with the cokernel of $\Hom(H,H) \hookrightarrow \Hom(H,V)$, which is $\Hom(H,V/H)$.

\eqref{it:gr-open}~As above, let $H'$ be a complement to $H$.  Identifying $V = H \oplus H'$, we have two projection maps $p: V \to H$ and $p': V \to H'$.  Let
\begin{align*}
W &= \{M \in \Gr_k(V) \mid M \cap H' = 0 \} \\
&= \{ M \in \Gr_k(V) \mid \text{$p|_M: M \to H$ is an isomorphism} \}.
\end{align*}
This is an open subset of $\Gr_k(V)$ that contains $H$.  For $M \in W$, let $\theta_M: H \to H'$ be the composition
\[
H \xrightarrow[\sim]{(p|_M)^{-1}} M \xrightarrow{p'|_M} H'.
\]
It is easily checked that $M \mapsto \theta_m$ defines an isomorphism $W \cong \Hom(H,H') = \Hom(H,V/H)$.
\end{proof}

In the setting of Lemma~\ref{lem:v4-decomp}, the subspace $H \subset V$ has a specified complement $M_3 \oplus M_4$, so Lemma~\ref{lem:gr-tan} lets us identify
\begin{equation}\label{eqn:gr-tan-decomp}
T_H\Gr_k(V) \cong \Hom(H, M_3 \oplus M_4) = \Hom(M_1 \oplus M_2, M_3 \oplus M_4).
\end{equation}
We can therefore write elements of this tangent space as $2 \times 2$ matrices.

\begin{lem}\label{lem:gr-tan-orbit}
Let $H \subset V$ be a subspace of dimension $k$, and choose a decomposition of $V$ as in Lemma~\ref{lem:v4-decomp}.  Let $\cC$ be the $\Iso(V)$-orbit of $H$ in $\Gr_k(V)$.  The tangent space to $\cC$ at $H$ is given by
\[
T_H\cC = 
\left\{
\begin{bmatrix}
a & b \\
c & d
\end{bmatrix} \in \Hom(M_1 \oplus M_2, M_3 \oplus M_4)
\,\Big|\,
c^\tra = - c
\right\}
\]
\end{lem}
\begin{proof}
From the proof of Lemma~\ref{lem:gr-tan}, the map $\fgl(V) \to T_H\Gr_k(V)$ is given by
\begin{equation}\label{eqn:tansp-map}
\left[\begin{smallmatrix}
a & b & c & d \\
e & f & g & h \\
i & j & k & l \\
m & n & p & q
\end{smallmatrix}\right]
\mapsto 
\left[\begin{smallmatrix}
i & j \\
m & n 
\end{smallmatrix}\right]
\end{equation}
To compute $T_H\cC$, we must compute the image of $\Lie(\Iso(V)) \subset \fgl(V)$ under this map.  Since $\Lie(\Iso(V)) = \{ \theta \in \fgl(V) \mid \theta^\tra = -\theta \}$, we see from~\eqref{eqn:v4-adjoint} that
\[
\Lie(\Iso(V))
=
\left\{
\left[\begin{smallmatrix}
a & b & c & d \\
-n^\tra & f & -j^\tra & -b^\tra \\
i & j & k & -c^\tra \\
m & n & -\tra & -a^\tra
\end{smallmatrix}
\right]
\, \Bigg|\,
\begin{array}{@{}c@{}}
\text{$a$, $b$, $c$, $i$, $j$, $n$ arbitrary, and} \\
\text{$d^\tra = -d$, $f^\tra = -f$, $k^\tra = -k$, and $m^\tra = -m$}
\end{array}
\right\}.
\]
The image of this under~\eqref{eqn:tansp-map} is the space described in the statement of the lemma.
\end{proof}

\begin{cor}\label{cor:gr-tan-slice}
Let $H \subset V$ be a subspace of dimension $k$, and choose a decomposition of $V$ as in Lemma~\ref{lem:v4-decomp}. Let $\cC$ be the $\Iso(V)$-orbit of $H$ in $\Gr_k(V)$.  There is a transverse slice $S$ to $\cC$ and $H$ and an isomorphism
\[
S \cong \{ c \in \Hom(M_1, M_4) \mid c^\tra = c \}
\]
that identifies $H \in S$ with $0 \in \Hom(M_1,M_4)$.

Moreover, there is a cocharacter $\psi: \Gm \to \Iso(V)$ such that the induced $\Gm$-action on $\Gr_k(V)$ preserves $S$, and acts on it linearly with strictly positive weights.
\end{cor}
\begin{proof}
It is clear from~\eqref{eqn:gr-tan-decomp} and Lemma~\ref{lem:gr-tan-orbit} that $\{ c \in \Hom(M_1, M_4) \mid c^\tra = c \}$ is a linear complement to $T_H \cC$ in $T_H \Gr_k(V)$.  Via Lemma~\ref{lem:gr-tan}\eqref{it:gr-open}, we can identify this subspace with a locally closed subvariety $S \subset \Gr_k(V)$ that meets $\cC$ transversely.

Next, let $\psi: \Gm \to \GL(V)$ be the map given by
\[
\psi(z)v = 
\begin{cases}
z^{-1}v & \text{if $v \in M_1$,} \\
zv & \text{if $v \in M_4$,} \\
v & \text{if $v \in M_2$ or $v \in M_3$.}
\end{cases}
\]
This map actually takes values in $\Iso(V) \subset \GL(V)$, and even in $\Stab_{\Iso(V)}(H)$ (by Lemma~\ref{lem:v4-stab}).  The induced action on $\Hom(M_1,M_4)$ is given by $z \cdot a = z^2a$.
\end{proof}

%--------------------------------------------------------------------------
\subsection{The symplectic case}
\label{ss:orbit-symp}
%--------------------------------------------------------------------------

We now specialize to the symplectic case.  In this subsection, assume that
\[
\lag{-},{-}\rag: V \times V \to \C
\]
is skew-symmetric.

\begin{prop}\label{prop:orbit-symp}
Suppose $V$ is equipped with nondegenerate skew-symmetric bilinear form.  For $0 \le k \le n$, let $\cC_{k,r}(V) \subset \Gr_k(V)$ be as in Lemma~\ref{lem:v4-orbit}.
\begin{enumerate}
\item The set $\cC_{k,r}(V)$ is nonempty if and only if\label{it:orbit-sympne}
\[
r \equiv 0 \pmod 2 \qquad\text{and}\qquad 
\max\{0, 2k - n\} \le r \le k.
\]
\item If $\cC_{k,r}(V)$ is nonempty, it is a single $\Sp(V)$-orbit.\label{it:orbit-single}
\item The stabilizer in $\Sp(V)$ of any point in $\cC_{k,r}(V)$ is connected.\label{it:orbit-stab}
%\item If $\cC_{k,r}(V)$ is nonempty, it is simply connected.\label{it:orbit-pi1}
\item For any commutative ring $\bk$, $\coh^\bullet_{\Sp(V)}(\cC_{k,r}(V),\bk)$ is even and free.\label{it:orbit-eqcoh}
\end{enumerate}
\end{prop}
\begin{proof}
For brevity, we write $\cC_{k,r}$ instead of $\cC_{k,r}(V)$ in this proof.  There is an obvious $\Sp(V)$-equivariant isomorphism $\Gr_k(V) \overset{\sim}{\to} \Gr_{n -k}(V)$ given by $H \mapsto H^\perp$.  This restricts to an isomorphism between $\cC_{k,r}$ and $\cC_{n-k,n-2k+r}$. By replacing $k$ by $n - k$ if necessary, we may as well assume that $k \le n/2$.  We will prove parts~\eqref{it:orbit-sympne}--\eqref{it:orbit-eqcoh} under this assumption.

\eqref{it:orbit-sympne}~When $k \le n/2$, this statement says that $\cC_{k,r}$ is nonempty if and only if $r$ is even and $0 \le r \le k$.  If $\cC_{k,r}$ is nonempty, then for any $H \in \cC_{k,r}$, the quotient $H/(H \cap H^\perp)$ inherits a nondegenerate symplectic form, so its dimension $r$ must be even. It is obvious that $0 \le r \le k$. 

For the opposite implication, choose a basis $\{x_1, \ldots, x_n\}$ for $V$ such that
\[
\lag x_i, x_j \rag =
\begin{cases}
1 & \text{if $i < j$ and $i + j = n + 1$,}\\
-1 & \text{if $i > j$ and $i + j = n + 1$,} \\
0 & \text{otherwise.}
\end{cases}
\]
Given an even number $r$ such that $0 \le r \le k$, we see that
\[
\spn\{x_1, \ldots, x_{k-r}\} \cup \{ x_{\frac{\dim V - r}{2}+1}, x_{\frac{\dim V - r}{2}+2}, \ldots, x_{\frac{\dim V + r}{2}} \}
\]
lies in $\cC_{k,r}$, so $\cC_{k,r}$ is nonempty.

\eqref{it:orbit-single}~This holds by Lemma~\ref{lem:v4-orbit}.

\eqref{it:orbit-stab}~By Lemma~\ref{lem:v4-stab}, we have
\[
\Stab_{\Sp(V)}(H) \cong (\GL(M_1) \times \Sp(M_2) \times \Sp(M_3)) \ltimes (\text{a unipotent group}),
\]
and this is connected.

%\eqref{it:orbit-pi1}~Since $\Sp(V)$ is connected and simply connected, the long exact sequence of homotopy groups associated to the fiber sequence
%\[
%\Stab_{\Sp(V)}(H) \hookrightarrow \Sp(V) \to \cC_{k,r}
%\]
%shows that $\pi_1(\cC_{k,r},H) \cong \pi_0(\Stab_{\Sp(V)}(H))$, so the claim follows from part~\eqref{it:orbit-stab}.

\eqref{it:orbit-eqcoh}~We have
\[
\coh^i_{\Sp(V)}(\cC_{k,r},\bk)
\cong \coh^i_{\Stab(H)}(\mathrm{pt},\bk)
\cong \coh^i_{\GL(M_1) \times \Sp(M_2) \times \Sp(M_3)}(\mathrm{pt},\bk),
\]
It is well known that the equivariant cohomology of a point vanishes in odd degrees as long as the torsion primes for the group are invertible in $\bk$: see, for instance,~\cite[Theorem~6.7.9]{a:book}.  But $\GL(M_1) \times \Sp(M_2) \times \Sp(M_3)$ has no torsion primes, by~\cite[\S4.4]{ss:cc}.
\end{proof}

\begin{rmk}\label{rmk:orbit-pi1}
Using the fact that $\Sp(V)$ is simply connected, one can deduce from Proposition~\ref{prop:orbit-symp}\eqref{it:orbit-stab} that each orbit $\cC_{k,r}(V)$ is also simply connected.
\end{rmk}

%\begin{cor}\label{cor:orbit-sym}
%For any commutative ring $\bk$, $\coh^\bullet_{\Sp(V)}(\Gr_k(V),\bk)$ is even and free.
%\end{cor}

%--------------------------------------------------------------------------
\subsection{The orthogonal case}
\label{ss:orbit-orth}
%--------------------------------------------------------------------------

In this subsection, assume that
\[
\lag{-},{-}\rag: V \times V \to \C
\]
is symmetric.  In contrast with the symplectic case, not all $\SO(V)$-orbits on the Grassmannian are simply connected: some admit $2$-to-$1$ covering maps.  For any commutative ring $\bk$, let
\[
\sgn = \sgn_\bk
\]
denote the $\bk[\Z/2\Z]$-module whose underlying $\bk$-module is just $\bk$ itself, and on which the nontrivial element of $\Z/2\Z$ acts by $-1$.  If $2 \ne 0$ in $\bk$, then $\sgn_\bk$ and the trivial $\bk[\Z/2\Z]$-module are not isomorphic.  In addition, if $2$ is invertible in $\bk$, then the regular representation decomposes as
\begin{equation}\label{eqn:reg-decomp}
\bk[\Z/2\Z] \cong \bk \oplus \sgn.
\end{equation}

\begin{prop}\label{prop:orbit-orth}
Suppose $V$ is equipped with nondegenerate symmetric bilinear form.  For $0 \le k \le n$, let $\cC_{k,r}(V) \subset \Gr_k(V)$ be as in Lemma~\ref{lem:v4-orbit}.
\begin{enumerate}
\item The set $\cC_{k,r}(V)$ is nonempty if and only if
$\max\{0, 2k - n\} \le r \le k$.\label{it:orbit-orthne}
\item If $\cC_{k,r}$ is nonempty, it is a single $\rO(V)$-orbit.  Moreover, if $n -2k + r > 0$ or $r > 0$, then $\cC_{k,r}(V)$ is a single $\SO(V)$-orbit.  If $n - 2k + r = r = 0$, then $\cC_{k,r}(V)$ is a union of two $\SO(V)$-orbits.\label{it:orbit-12}
\item Let $H \in \cC_{k,r}(V)$.  We have\label{it:orbit-orthstab}
\[
\Stab_{\SO(V)}(H)/ \Stab_{\SO(V)}(H)^\circ \cong
\begin{cases}
\Z/2\Z & \text{if $\max \{0,2k-n\} < r$,} \\
1 & \text{if $\max \{0,2k-n\} = r$.}
\end{cases}
\]
%In addition, if $n \ge 3$, we have
%\[
%\Stab_{\Spin(V)}(H)/\Stab_{\Spin(V)}(H)^\circ \cong
%\Stab_{\SO(V)}(H)/ \Stab_{\SO(V)}(H)^\circ.
%\]
%\item Suppose $\cC_{k,r}$ is nonempty.  For any $H \in \cC_{k,r}$, we have\label{it:orbit-orthp1}
%\[
%\pi_1(\cC_{k,r},H) \cong
%\begin{cases}
%\Z/2\Z & \text{if $n \ge 3$ and $\max \{0,2k-n\} < r$,} \\
%\Z & \text{if $n = 2$ and $\max \{0,2k-n\} < r$,} \\
%1 & \text{if $\max \{0,2k-n\} = r$.}
%\end{cases}
%\]
%As a consequence, if $n \ne 3$, every local system on (a component of) $\cC_{k,r}$ is automatically $\SO(V)$-equivariant.
\item Let $\cC$ be an $\SO(V)$-orbit on $\Gr_k(V)$, and let $\bk$ be a commutative ring in which $2$ is invertible.  Then $\coh^\bullet_{\SO(V)}(\cC,\bk)$ is even and free.\label{it:orbit-ortheqcoh}
In the case where points of $\cC$ have disconnected stabilizers, $\coh^\bullet_{\SO(V)}(\cC, \sgn)$ is also even and free.
\end{enumerate}
\end{prop}
When $n$ is even, we (arbitrarily) label the two orbits in $\cC_{n/2,0}(V)$ by 
\[
\cC_{n/2,0'}(V) \qquad\text{and}\qquad \cC_{n/2,0''}.
\]
\begin{proof}
Let us write $\cC_{k,r}$ for $\cC_{k,r}(V)$.  There is an obvious $\SO(V)$-equivariant isomorphism $\Gr_k(V) \overset{\sim}{\to} \Gr_{n -k}(V)$ given by $H \mapsto H^\perp$.  This restricts to an isomorphism between $\cC_{k,r}$ and $\cC_{n-k,n-2k+r}$. By replacing $k$ by $n - k$ if necessary, we may as well assume that $k \le n/2$.  We will prove parts~\eqref{it:orbit-orthne}--\eqref{it:orbit-ortheqcoh} under this assumption.

\eqref{it:orbit-orthne}~When $k \le n/2$, this statement says that $\cC_{k,r}$ is nonempty if and only if $0 \le r \le k$.  The ``only if'' direction is obvious.  For the ``if'' direction, choose an orthonormal basis $\{x_1, \ldots, x_n\}$ for $V$, i.e., a basis satisfying
\[
\lag x_i, x_j\rag = 
\begin{cases}
1 & \text{if $i = j$,} \\
0 & \text{otherwise.}
\end{cases}
\]
For $0 \le r \le k$, the subspace
\begin{multline*}
\spn \{x_1 + \sqrt{-1} x_2, x_3 + \sqrt{-1} x_4, \ldots, x_{2k-2r-1} + \sqrt{-1} x_{2k-2r}\} \\ {}\cup \{ x_{2k-2r+1}, x_{2k-2r+2}, \ldots, x_{2k-r} \}
\end{multline*}
lies in $\cC_{k,r}$, so $\cC_{k,r}$ is nonempty.

\eqref{it:orbit-12}~All but the case where $n-2k+r = r = 0$ is covered by Lemma~\ref{lem:v4-orbit}.  We will return to this case after proving part~\eqref{it:orbit-orthstab}.

\eqref{it:orbit-orthstab}~By Lemma~\ref{lem:v4-stab}, we have
\[
\Stab_{\rO(V)}(H) \cong (\GL(M_1) \times \rO(M_2) \times \rO(M_3)) \ltimes (\text{a unipotent group}).
\]
Write an element of this group as a quadruple $(g_1,g_2,g_3,v)$.  The set of connected components of $\Stab_{\rO(V)}(H)$ is in bijection with the set of possible values of the pair $(\det(g_2),\det(g_3))$, and thus depends on whether $M_2$ and $M_3$ are zero or not.  A connected component lies in $\SO(V)$ if $\det(g_2)\det(g_3) = 1$.  Recall that $\dim M_2 = r$ and $\dim M_3 = n-2k+r$.  The following table summarizes the possible cases.
\[
\begin{array}{c@{\quad}c@{\quad}c}
& \text{Number of components in}
& \text{Number of components in} \\
\text{Condition} & \Stab_{\rO(V)}(H) & \Stab_{\SO(V)}(H) \\
\hline
\max\{0,2k-n\} < r & 4 & 2 \\
2k-n < 0 = r & 2 & 1 \\
2k-n = r = 0 & 1 & 1
\end{array}
\]
(The case $0 < 2k-n = r$ cannot occur under the assumption that $k \le n/2$.)  This table shows that the group $\Stab_{\SO(V)}(H)/\Stab_{\SO(V)}(H)^\circ$ is as claimed in the proposition.

%Suppose now that $n \ge 3$.  To show $\Stab_{\Spin(V)}(H)/\Stab_{\Spin(V)}(H)^\circ$ has the same form, we must show that the preimage of $\Stab_{\SO(V)}(H)^\circ$ under $\Spin(V) \to \SO(V)$ is connected.  Lemma~\ref{lem:spin-conn} below gives a criterion for this to hold.  The following table lists pairs of vectors we can use to invoke that lemma.
%\[
%\begin{array}{c@{\quad}c}
%& \text{Where to find vectors $(x,y)$ satisfying the} \\
%\text{Condition} & \text{hypotheses of Lemma~\ref{lem:spin-conn}} \\
%\hline
%k > r  & x \in M_1,\ y \in M_4 \\
%k = r,\ r \ge 2 & x, y \in M_2 \\
%k = r ,\ r \le 1 & x, y \in M_3
%\end{array}
%\]

\eqref{it:orbit-12} for $r  = 0$ and $2k-n= 0$.  We must show that $\cC_{k,r}$ consists of two $\SO(V)$-orbits in this case.  We have already seen that it is a single $\rO(V)$-orbit, so it consists of at most two $\SO(V)$-orbits.  On the other hand, we saw above that $\Stab_{\rO(V)}(H)$ is connected for $H \in \cC_{k,r}$, so $\cC_{k,r}$ has two connected components, and therefore must consist of two $\SO(V)$-orbits.

%\eqref{it:orbit-orthp1}~Suppose first that $n \ge 3$.  Since $\Spin(V)$ is simply connected, the same reasoning as in Proposition~\ref{prop:orbit-orth}\eqref{it:orbit-pi1} shows that $\pi_1(\cC_{k,r},H) \cong \pi_0(\Stab_{\Spin(V)}(H))$, so the claim follows from part~\eqref{it:orbit-orthstab}.
%
%Next, suppose $n \le 2$.  If $k = 0$ or $k = n$, then we necessarily have $r = k$ (so $r = \max \{0,2k-n\}$).  In this case, $\cC_{k,r} = \Gr_k(V)$ is a singleton, and thus simply connected.
%
%The remaining cases are those in which $n = 2$, $k = 1$.  In this case, $\Gr_k(
%V) \cong \bP^1$.  In the $\SO(V)$-action on $\Gr_k$, there are two singleton orbits, corresponding to the two isotropic lines in $V$.  The singletons are the connected components of $\cC_{1,0}$; they are simply connected.  The remaining orbit is $\cC_{1,1} \cong \bP^1 \smallsetminus \{\text{two points}\}$; its fundamental group is $\Z$.
%

\eqref{it:orbit-ortheqcoh}~Let $\cC$ be an $\SO(V)$-orbit, and let $H \in \cC$. Suppose first that $\Stab_{\SO(V)}(H)$ is connected. Then
\[
\coh^\bullet_{\SO(V)}(\cC,\bk)
\cong \coh^\bullet_{(\GL(M_1) \times \rO(M_2) \times \rO(M_3))^\circ}(\mathrm{pt},\bk).
\]
By~\cite[\S4.4]{ss:cc}, the only torsion prime for the reductive group $(\GL(M_1) \times \rO(M_2) \times \rO(M_3))^\circ$ is $2$.  By assumption, $2$ is invertible in $\bk$, so the cohomology groups $\coh^\bullet_{(\GL(M_1) \times \rO(M_2) \times \rO(M_3))^\circ}(\mathrm{pt},\bk)$ are even and free.

Suppose now that $\Stab_{\SO(V)}(H)$ is disconnected, and let
\[
\eC = \SO(V)/\Stab_{\SO(V)}(H)^\circ.
\]
Then we have a $2$-to-$1$ covering map $p: \eC \to \cC$.  Since $2$ is invertible in $\bk$, we have $p_*\bk \cong \bk \oplus \sgn$, so it is enough to show that $\coh^\bullet_{\SO(V)}(E,\bk)$ is even and free.  But
\[
\coh^\bullet_{\SO(V)}(\eC, \bk) 
\cong \coh^\bullet_{(\GL(M_1) \times \rO(M_2) \times \rO(M_3))^\circ}(\mathrm{pt},\bk),
\]
and this is even and free by the same reasoning as in the previous paragraph.
\end{proof}

\begin{rmk}\label{rmk:orbit-orthp1}
In contrast with the setting of Remark~\ref{rmk:orbit-pi1}, the fundamental group of $\cC_{k,r}(V)$ can be more complicated.  One can show that
\[
\pi_1(\cC_{k,r}(V),H) \cong
\begin{cases}
\Z/2\Z & \text{if $n \ge 3$ and $\max \{0,2k-n\} < r$,} \\
\Z & \text{if $n = 2$ and $\max \{0,2k-n\} < r$,} \\
1 & \text{if $\max \{0,2k-n\} = r$.}
\end{cases}
\]
We will not need these calculations in this paper, however.
\end{rmk}

\section{Some results on orbit closures}
\label{sec:resoln}
%%%%%%%%%%%%%%%%%%%%%%%%%%%%%%%%%%%%%%%%%%%%%%%%%%%%%%%%%%%%%%%%%%%%%%%%%%%

This section contains two results related to the geometry of orbit closures in the setting of Section~\ref{ss:orbit-symp} or~\ref{ss:orbit-orth}.  The first exhibits a resolution of singularities.  The second deals with non-simply-connected orbits in the orthogonal case.

\begin{prop}\label{prop:m1-resoln}
Let $V$ be an $n$-dimensional complex vector space equipped with a nondegenerate symmetric or skew-symmetric bilinear form.  Let $0 \le k \le n$, and let $r$ be such that $\max\{0, 2k-n\} \le r \le k$.  In the skew-symmetric case, assume also that $r$ is even.  Let
\[
\tC_{k,r}(V) = \{ P \subset H \subset P^\perp \subset V \mid \text{$\dim P = k -r$ and $\dim H = k$} \}.
\]
Let $\varpi: \tC_{k,r}(V) \to \Gr_k(V)$ be the map given by
\[
\varpi(P,H) = H.
\]
Then the image of $\varpi$ is $\overline{\cC_{k,r}(V)}$, and the map $\varpi: \tC_{k,r}(V) \to \overline{\cC_{k,r}(V)}$ is a resolution of singularities.
\end{prop}
\begin{proof}
	If $H$ is in the image of $\varpi$, then it satisfies the condition $P\subset H\subset P^\perp$ with $\dim P= k-r$. This implies $P\subset H^\perp\cap H$. Hence $k-r\leq \dim H\cap H^\perp $ and $H\in \overline{\cC_{k,r}(V)}$.
	The map $\varpi|_{\varpi^{-1}(\cC_{k,r}(V))}: \varpi^{-1}(\cC_{k,r}(V))\to \cC_{k,r}(V)$ is clearly an isomorphism. Now consider the map, $\tC_{k,r}\to \Gr_{k-r}(V)^\iso$ by sending $(P,H)$ to $P$. The fibers of this map are isomorphic to $\{H \mid P\subset H \subset P^\perp\}\cong \Gr_r(P^\perp/P)$, which is smooth. Also $k-r< k-r/2 \leq n/2$, hence by Lemma \ref{lem:griso-conn}, $\Gr_{k-r}(V)^\iso$ is smooth. Therefore $\tC_{k,r}(V)$ is smooth. We know the projection map $\Gr_{k-r}(V)^\iso \times \Gr_r(V)\to \Gr_k(V)$ is proper and $\tC_{k,r}(V)$ is a closed subset of $\Gr_{k-r}(V)^\iso \times \Gr_r(V) $. Hence  the map, $\varpi$ is proper. Combining these facts, we deduce that $\varpi$ is a resolution of singularities.
\end{proof}

For the remainder of this section, we work in the symmetric case, i.e., in the setting of Section~\ref{ss:orbit-orth}.  Suppose $\max \{0, 2k-n\} < r$, so that $\cC_{k,r}(V)$ is not simply connected.  We have already seen in the proof of Proposition~\ref{prop:orbit-orth}\eqref{it:orbit-ortheqcoh} how to construct a certain $2$-to-$1$ covering map $p: \eC \to \cC_{k,r}(V)$.

It would have been useful later in the paper to have a ``compactification'' of $p$, i.e., a proper map $p': \eC' \to \overline{\cC_{k,r}(V)}$ with smooth domain $\eC'$ such that the restriction $p'|_{(p')^{-1}(\cC_{k,r}(V))}: (p')^{-1}(\cC_{k,r}(V)) \to \cC_{k,r}(V)$ is identified with the covering map $p$.  Unfortunately, we do not know how to construct such a compactification of the $2$-to-$1$ covering $p$ explicitly.

As a substitute, we will construct a compactification of a ``$\pi_1$-injective fibration of index $2$,'' by which we mean a locally trivial fibration $F \to \cC_{k,r}$ in which $F$ is connected, and the induced map
\[
\pi_1(F) \to \pi_1(\cC_{k,r})
\]
is injective.  More precisely, we will construct a proper map
\[
\widehat{\varpi}: \hC_{k,r}(V) \to \overline{\cC_{k,r}(V)}
\]
whose domain is smooth, and whose restriction to the preimage of $\cC_{k,r}(V)$ is a $\pi_1$-injective fibration.

We begin by defining $\hC_{k,r}(V)$.  The definition is different depending on whether $r$ is even or odd.  If $r$ is even, we set  
\[
\hC_{k,r}(V) = \left\{ 
\begin{array}{@{}l@{}}
P \subset Q \subset H \\ {}\quad\subset P^\perp \subset V 
\end{array}\,\Big|\,
\begin{array}{@{}c@{}}
\text{$\dim P = k - r$, $\dim Q = k - \frac{r}{2}$,}\\
\text{$\dim H = k$, and $Q \subset Q^\perp$}
\end{array} \right\}
\qquad\text{($r$ even).}
\]
Before defining it in the odd case, we need some additional notation.  Let
\[
V^+ = V \oplus \C,
\]
and equip it with the nondegenerate symmetric bilinear form $\lag{-},{-}\rag_+$ given by
\[
\lag (v,z), (v',z')\rag_+ = \lag v,v'\rag + zz'
\qquad\text{for $v, v' \in V$ and $z, z' \in \C$.}
\]
There is an obvious injective map
\[
\cC_{k,r}(V)  \to \cC_{r+1,k+1}(V^+)
\qquad\text{given by}\qquad H \mapsto H^+ = H \oplus \C.
\]
We now set
\[
\hC_{k,r}(V) = \left\{ 
\begin{array}{@{}c@{}}
P \subset H \subset P^\perp \subset V,\\ P \subset Q \subset H^+ \subset V^+
\end{array}\,\Big|\,
\begin{array}{@{}c@{}}
\text{$\dim P = k - r$, $\dim Q = k - \frac{r-1}{2}$,}\\
\text{$\dim H = k$, and $Q \subset Q^\perp_+$}
\end{array} \right\}
\qquad\text{($r$ odd).}
\]
Here $Q^\perp_+$ is the orthogonal complement of $Q$ with respect to $\lag{-},{-}\rag_+$.

Before studying $\hC_{k,r}$ in detail, we define another variety $\cB_{k,r}(V)$ as follows:
\begin{align*}
\cB_{k,r}(V) &= \left\{ P \subset Q \subset V \,\Big|\,
\begin{array}{@{}c@{}}
\text{$\dim P = k - r$, $\dim Q = k - \frac{r}{2}$,}\\
\text{and $Q \subset Q^\perp$}
\end{array} \right\}&&
\text{if $r$ even,}
\\
\cB_{k,r}(V) &= \left\{ \begin{array}{@{}c@{}}Q \subset V^+, \\ P \subset Q \cap V \subset V \end{array} \,\Big|\,
\begin{array}{@{}c@{}}
\text{$\dim P = k - r$, $\dim Q = k - \frac{r-1}{2}$,}\\
\text{and $Q \subset Q^\perp_+$}
\end{array} \right\}&&
\text{if $r$ odd.}
\end{align*}

There are obvious maps
\[
\hC_{k,r}(V) \to \cB_{k,r}(V),
\qquad
\hC_{k,r}(V) \to \cC_{k,r}(V)
\]
given by forgetting $H$ or $Q$, respectively.  Let
\[
\widehat{\varpi}: \hC_{k,r}(V) \to \overline{\cC_{k,r}(V)}
\]
denote the composition $\hC_{k,r}(V) \to \tC_{k,r}(V) \xrightarrow{\varpi} \overline{\cC_{k,r}(V)}$.  Next, let
\[
\hC_{k,r}^\circ(V) = (\widehat{\varpi})^{-1}(\cC_{k,r}(V))
\qquad\text{and}\qquad
\widehat{\varpi}^\circ = \widehat{\varpi}|_{\hC_{k,r}^\circ(V)}: \hC_{k,r}^\circ(V) \to \cC_{k,r}(V).
\]

\begin{prop}\label{prop:fund-ses}
Let $V$ be an $n$-dimensional vector space equipped with a nondegenerate symmetric bilinear form, and let $0 \le k \le n$.  Let $r$ be an integer such that $\max \{0,2k-n\} < r \le k$.
\begin{enumerate}
\item The variety $\cB_{k,r}(V)$ is smooth, connected, and simply connected.\label{it:barbkr}
\item For any commutative ring $\bk$, $\coh^\bullet(\cB_{k,r}(V),\bk) = 0$ is even and free.\label{it:bkr-coh}
\item The variety $\hC_{k,r}$ is smooth, connected, and simply connected.\label{it:hckr}
\item The map $\widehat{\varpi}^\circ: \hC_{k,r}^\circ(V) \to \cC_{k,r}$ is a locally trivial fibration.\label{it:hfib}
%\item The variety $\hC_{k,r}^\circ(V)$ is smooth and connected.  Every $\SO(V)$-equivariant local system on 
%induced by $\widehat{\varpi}^\circ$ is injective, and its image is a subgroup of index~$2$.\label{it:pi-ses}
\item Let $\bk$ be a commutative ring in which $2$ is invertible.  There is an isomorphism of local systems \label{it:push-forward}
\[
\cH^0(\widehat{\varpi}_*^\circ\bk) \cong \bk \oplus \sgn.
\]
\end{enumerate}
\end{prop}
\begin{proof}
\eqref{it:barbkr}  There is a map
\begin{equation}\label{eqn:barbkr}
p: \cB_{k,r}(V) \to \Gr_{k-r}(V)^\iso
\end{equation}
given by $(P,Q) \mapsto P$.  Since $2k -n < r$, we have $k-r < k - r/2 < n/2$, and so $\Gr_{k-r}(V)^\iso$ is smooth, connected, and simply connected by Lemma~\ref{lem:griso-conn}.  

Suppose for now that $r$ is even. Let $\mc{F}l(k-r,k-r/2)$ be the variety of partial flags $P \subset Q \subset V$ with $\dim P = k-r$ and $\dim Q = k-r/2$.  The projection $\mc{F}l(k-r,k-r/2) \to \Gr_{k-r}(V)$ is a locally trivial fibration, and the map $p$ in~\eqref{eqn:barbkr} is the restriction of this map to the preimage of $\Gr_{k-r}(V)^\iso \subset \Gr_{k-r}(V)$. Hence $p$ is a locally trivial fibration whose fiber over a point $P \in \Gr_{k-r}(V)^\iso$ is the variety
\[
\{ Q \mid \text{$\dim Q = k - r/2$ and $P \subset Q \subset Q^\perp \subset P^\perp$} \} \cong \Gr_{r/2}(P^\perp / P)^\iso.
\]
Since $2k -n < r$, we have $\dim P^\perp /P = n - 2(k - r) = n - 2k + 2r > r$, so $\Gr_{r/2}(P^\perp/P)^\iso$ is again smooth, connected, and simply connected.  We conclude that $\cB_{k,r}(V)$ is smooth, connected, and simply connected.

The case where $r$ is odd is similar, but the fiber of $p$ over $P \in \Gr_{k-r}(V)^\iso$ is given by
\[
\{ Q \mid \text{$\textstyle\dim Q = k - \frac{r-1}{2}$ and $P \subset Q \subset Q^\perp_+ \subset P^\perp \oplus \C$} \} \cong \Gr_{\frac{r+1}{2}}(P^\perp / P \oplus \C)^\iso.
\]
Since $\dim (P^\perp/P \oplus \C )> r+1$, the variety $\Gr_{\frac{r+1}{2}}(P^\perp / P \oplus \C)^\iso$ is smooth, connected, and simply connected, and we conclude as before.

\eqref{it:bkr-coh}
We have seen in part~\eqref{it:barbkr} that $\cB_{k,r}$ is an isotropic Grassmannian bundle over an isotropic Grassmannian.  The claim that its cohomology is even and free then follows from the Serre spectral sequence and Corollary~\ref{cor:gr-oddvan}.

\eqref{it:hckr}
Suppose first $r$ is even. From the definitions, the map $\hC_{k,r}(V) \to \cB_{k,r}(V)$ is a locally trivial fibration whose fiber over a point $(P,Q) \in \cB_{k,r}(V)$ is identified with
\[
\Gr_{r/2}(P^\perp/Q),
\]
which is smooth, connected, and simply connected.  In view of part~\eqref{it:barbkr}, the claim for $\hC_{k,r}(V)$ follows.

Now suppose that $r$ is odd.  For a point $(P,Q) \in \cB_{k,r}(V)$, let $\bar Q$ be the image of $Q$ under the projection map $V^+ \to V$.  The kernel of this projection map contains no nonzero isotropic vectors.  Since $Q$ is isotropic, it follows that $\dim \bar Q = \dim Q = k - \frac{r-1}{2}$.  Note also that $Q$ is contained in $P^\perp_+ = P^\perp \oplus \C$, and hence that $\bar Q \subset P^\perp$. Finally, in the definition of $\hC_{k,r}(V)$, the condition $Q \subset H^+ \subset P^\perp_+ \subset V^+$ is equivalent to $\bar Q \subset H \subset P^\perp \subset V$.  In view of these observations, we deduce that $\hC_{k,r}(V) \to \cB_{k,r}(V)$ is a locally trivial fibration whose fiber over a point $(P,Q) \in \cB_{k,r}(V)$ is identified with
\[
\Gr_{\frac{r-1}{2}}(P^\perp/\bar Q),
\]
and then we conclude as before.

\eqref{it:hfib}
Let $H \in \cC_{k,r}(V)$, and let $(P,Q,H) \in \widehat{\varpi}^{-1}(H)$.  The definition of $\hC_{k,r}(V)$ implies that $P \subset \rad H$, but the fact that $H \in \cC_{k,r}$ means that $\dim \rad H = k - r = \dim P$.  These observations let us describe $\hC_{k,r}^\circ(V)$ as follows: if $r$ is even, then
\begin{align*}
\hC_{k,r}^\circ(V) &= \left\{ 
\rad H \subset Q \subset H \,\Big|\,
\begin{array}{@{}c@{}}
\text{$H \in \cC_{k,r}(V)$, $\dim Q = k - \frac{r}{2}$,}\\
\text{and $Q \subset Q^\perp$}
\end{array} \right\}
&&\text{if $r$ is even,} \\
\hC_{k,r}^\circ(V) &= \left\{ 
\rad H \subset Q \subset H^+ \,\Big|\,
\begin{array}{@{}c@{}}
\text{$H \in \cC_{k,r}(V)$, $\dim Q = k - \frac{r-1}{2}$,}\\
\text{and $Q \subset Q^\perp_+$}
\end{array} \right\}
&&\text{if $r$ is odd.}
\end{align*}
In both cases, it follows that $\hC_{k,r}^\circ(V) \to \cC_{k,r}$ is a locally trivially fibration.  The fiber over $H \in \cC_{k,r}$ is given by
\[
\Gr_{r/2}(H/\rad H)^\iso\quad\text{if $r$ is even,}
\qquad\qquad
\Gr_{\frac{r+1}{2}}(H^+/\rad H)^\iso\quad\text{if $r$ is odd.}
\]

%\eqref{it:pi-ses}
\eqref{it:push-forward}~
Since $\hC_{k,r}(V)$ is smooth and connected, its Zariski-open subset $\hC_{k,r}^\circ(V)$ is as well.  Let $\Gamma = \pi_1(\cC_{k,r}(V))$, and let $\Gamma' = \pi_1(\hC^\circ_{k,r}(V))$.  (The group $\Gamma$ is described in Remark~\ref{rmk:orbit-orthp1}, but we will not need this information.)

Next, let $F$ denote a fiber of $\widehat{\varpi}: \hC^\circ_{k,r}(V) \to \cC_{k,r}(V)$.  We have identified $F$ in the proof of part~\eqref{it:hfib}.  By Proposition~\ref{prop:orbit-orth}\eqref{it:orbit-12}, $F$ has two connected components, and by Lemma~\ref{lem:griso-conn}, each connected component is simply connected.  In other words, $\pi_0(F) \cong \Z/2\Z$ and $\pi_1(F) = 1$. Then the long exact sequence of homotopy groups
\[
\cdots \to \pi_1(F) \to \underbrace{\pi_1(\hC^\circ_{k,r}(V))}_{\Gamma'} \to \underbrace{\pi_1(\cC_{k,r}(V))}_{\Gamma} \to \pi_0(F) \to \pi_0(\hC^\circ_{k,r}(V)) \to \cdots.
\]
shows that $\Gamma' \to \Gamma$ is injective, and that its image is a subgroup of index~$2$.

The category of local systems on $\cC_{k,r}(V)$, resp.~on $\hC^\circ_{k,r}(V)$, is equivalent to the category of $\bk[\Gamma]$-modules, resp.~of $\bk[\Gamma']$-modules.  Under these equivalences, the pullback functor $\widehat{\varpi}^*$ corresponds to the forgetful functor
\begin{equation}\label{eqn:pullback-forget}
\bk[\Gamma]\text{-mod} \to \bk[\Gamma']\text{-mod}.
\end{equation}
Since $\widehat{\varpi}^\circ$ is a locally trivial fibration, the (non-derived) push-forward functor $\cH^0 \circ \widehat{\varpi}_*$ sends local systems to local systems.  It is right-adjoint to $\widehat{\varpi}^*$, so we can identify it with the right adjoint to~\eqref{eqn:pullback-forget}, which is
\[
\Hom_{\bk[\Gamma']}(\bk[\Gamma]^{\mathrm{op}}, {-}): \bk[\Gamma']\text{-mod} \to \bk[\Gamma]\text{-mod}.
\]
Therefore, the local system $\cH^0(\widehat{\varpi}_*\bk)$ corresponds to $\Hom_{\bk[\Gamma']}(\bk[\Gamma]^{\mathrm{op}},\bk)$, which can be identified with the regular representation of $\bk[\Gamma/\Gamma'] = \bk[\Z/2\Z]$.  This local system decomposes as in~\eqref{eqn:reg-decomp} because $2$ is invertible in $\bk$.
\end{proof}

%%%%%%%%%%%%%%%%%%%%%%%%%%%%%%%%%%%%%%%%%%%%%%%%%%%%%%%%%%%%%%%%%%%%%%%%%%%
\section{Direct sums of bilinear forms}
\label{sec:dsums}
%%%%%%%%%%%%%%%%%%%%%%%%%%%%%%%%%%%%%%%%%%%%%%%%%%%%%%%%%%%%%%%%%%%%%%%%%%%

Let $B_1, \ldots, B_m$ be a collection of nonzero finite-dimensional complex vector spaces.  Assume that each is equipped with nondegenerate bilinear form
\[
\lag{-},{-}\rag_{B_i}: B_i \times B_i \to \C
\]
that is either symmetric or skew-symmetric.  Let $\Iso(B_i) \subset \GL(B_i)$ be the group of linear automorphisms of $B_i$ that preserve $\lag{-},{-}\rag_{B_i}$, and let $\Iso(B_i)^\circ$ be its identity component.  That is, $\Iso(B_i)^\circ$ is either $\SO(B_i)$ or $\Sp(B_i)$, depending on $\lag{-},{-}\rag_{B_i}$.  Let

\[
I_B = \Iso(B_1)^\circ \times \Iso(B_2)^\circ \times \cdots \times \Iso(B_m)^\circ.
\]
We let
\[
n_i = \dim B_i
\qquad\text{and}\qquad n = n_1 + \cdots + n_m.
\]

Fix an integer $k$ such that $0 \le k \le n$.  Let $\uk = (k_1, \ldots, k_m)$ be a tuple of nonnegative integers satisfying $k_1 + \cdots + k_m = k$.  Let $\ur = (r_1, \ldots, r_m)$ be a sequence of symbols satisfying the following constraints:
\begin{itemize}
\item If $\lag{-},{-}\rag_{B_i}$ is skew-symmetric, then $r_i$ is an integer such that $r_i \equiv 0 \pmod 2$ and $\max \{0, 2k_i - n_i\} \le r_i \le k_i$.
\item If $\lag{-},{-}\rag_{B_i}$ is symmetric and $\dim B_i$ is odd, then $r_i$ is an integer such that $\max \{0, 2k_i - n_i\} \le r_i \le k_i$.
\item If $\lag{-},{-}\rag_{B_i}$ is symmetric and $\dim B_i$ is even, then $r_i$ is either an integer such that $\max \{1, 2k_i - n_i\} \le r_i \le k_i$, or it may be one of the two special symbols $0'$ or $0''$.
\end{itemize}
Let
\[
\Omega_k = \{ \text{pairs of $m$-tuples $(\uk,\ur)$ satisfying the conditions above} \}.
\]
Then Propositions~\ref{prop:orbit-symp} and~\ref{prop:orbit-orth} together give us a bijection
\[
\Omega_k \overset{\sim}{\leftrightarrow}
\bigsqcup_{\substack{k_1, \ldots, k_m \ge 0\\ k_1 + \cdots + k_m = k}} \{\text{$I_B$-orbits on $\Gr_{k_1}(B_1) \times \cdots \Gr_{k_m}(B_m)$}\}.
\]
For $(\uk,\ur) \in \Omega_k$, let
\[
\cC_{\uk,\ur}(B) = \cC_{k_1,r_1}(B_1) \times \cdots \times \cC_{k_m,r_m}(B_m)
\]
be the corresponding orbit.

Next, let
\begin{align*}
B_{\le i}&= B_1 \oplus B_2 \oplus \cdots \oplus B_i, \\
B & = B_1 \oplus B_2 \cdots \oplus B_i \oplus \cdots \oplus B_m.
\end{align*}
We also sometimes write $B_{<i} = B_{\le i-1}$.  There is a canonical identification
\[
B_i \cong B_{\le i}/B_{<i}.
\]
For a subspace $H \subset B$, we set
\[
\pr_i H = (H \cap B_{\le i})/(H \cap B_{<i}) \subset B_i.
\]

Let $U_B$ be the subgroup of $\GL(B)$ given by
\[
U_B = \{ g \in \GL(B) \mid \text{for $v \in B_i$, $g(v) \in v + B_{<i}$} \}.
\]
This is a unipotent group.  
We regard $I_B$ as a subgroup of $\GL(B)$ in the obvious way.  This group normalizes $U_B$, and we set
\[
Q_B = I_B \ltimes U_B \subset \GL(B).
\]
Occasionally, we will need to consider slightly larger groups. As in Remark~\ref{rmk:giso-defn}, we set
\[
\GIsoc(B_i) = \C^\times \times \Isoc(B_i).
\]
We also set
\begin{align*}
GI_B &= \GIsoc(B_1) \times \GIsoc(B_2) \times \cdots \times \GIsoc(B_m), \\
GQ_B &= GI_B \ltimes U_B.
\end{align*}

In part~\eqref{it:filt-oddvan} of the following proposition, a ``locally free $\bk$-local system of finite rank'' means a local system of $\bk$-modules whose stalks are finite-rank free $\bk$-modules.

\begin{prop}\label{prop:filt-orbits}
Let $0 \le k \le n$.  For $(\uk,\ur) \in \Omega_k$, let
\[
\cO_{\uk,\ur} = \{ H \in \Gr_k(B) \mid \text{for all $i$, $\dim \pr_i H = k_i$ and $\pr_i H \in \cC_{k_i,r_i}(B_i)$} \}.
\]
\begin{enumerate}
\item Each $\cO_{\uk,\ur}$ is a single $Q_B$-orbit and a single $GQ_B$-orbit.  Moreover, this construction establishes bijections\label{it:filt-orbit}
\[
\Omega_k \overset{\sim}{\leftrightarrow} \{\text{$Q_B$-orbits on $\Gr_k(B)$} \}.
= \{\text{$GQ_B$-orbits on $\Gr_k(B)$} \}
\]
\item The natural map\label{it:filt-bdle}
\[
\cO_{\uk,\ur} \to \cC_{\uk,\ur}
\]
given by $H \mapsto (\pr_1 H, \ldots, \pr_m H)$ makes $\cO_{\uk,\ur}$ into an affine space bundle over $\cC_{\uk,\ur}$.
\item Let $H \in \cO_{\uk,\ur}$.  We have\label{it:filt-p1}
\[
\Stab_{Q_B}(H)/\Stab_{Q_B}(H)^\circ \cong \Stab_{GQ_B}(H)/\Stab_{GQ_B}(H)^\circ\cong (\Z/2\Z)^d
\]
where
\[
d = |\{ i \mid \text{$\lag{-},{-}\rag_{B_i}$ is symmetric and $\max\{0,2k_i - n_i\} < r_i$} \}.
\]
\item Let $\bk$ be a commutative ring in which $2$ is invertible, and let $\cL$ be locally free $Q_B$-equivariant $\bk$-local system of rank~$1$ on $\cO_{\uk,\ur}$.  Then $\coh^\bullet_{Q_B}(\cO_{\uk,\ur},\cL)$ is even and free.\label{it:filt-oddvan}
\end{enumerate}
\end{prop}

\begin{proof}
We will first prove~\eqref{it:filt-bdle} by induction on $m$.  If $m = 1$, then $\cO_{\uk,\ur} = \cC_{\uk,\ur}$, and there is nothing to prove.  Otherwise, let $B' = B_{\le m-1}$, and let $n' = n_1 + \cdots + n_{m-1}$.  Given $(\uk,\ur) = ((k_1,\ldots, k_m),(r_1,\ldots, r_m))$, let $k' = k_1 + \cdots + k_{m-1}$, and let
\[
(\uk',\ur') = ((k_1,\ldots,k_{m-1}),(r_1,\ldots,r_{m-1})).
\]
This pair labels a subset $\cO_{\uk',\ur'} \subset \Gr_{k'}(B')$.  The map $\cO_{\uk,\ur} \to \cC_{\uk,\ur}$ in the statement of the proposition factors as
\begin{multline}\label{eqn:afffib}
\cO_{\uk,\ur} \xrightarrow{H \mapsto (H \cap B', \pr_m H)} 
\cO_{\uk',\ur'} \times \cC_{k_m,r_m}(B_m) \\
\xrightarrow{(H',K) \mapsto (\pr_1 H', \ldots, \pr_{m-1} H',K)} \cC_{k_1,r_1}(B_1) \times \cdots \times \cC_{k_m,r_m}(B_m) = \cC_{\uk,\ur}.
\end{multline}
The second map in~\eqref{eqn:afffib} an affine space bundle by induction, so it is enough to prove that the first map is an affine space bundle.

Let $p: B \to B'$ be the projection map with respect to the decomposition $B = B' \oplus B_m$.  For any $H \in \cO_{\uk,\ur}$, the map $p|_H: H \to B'$ induces a map $\bar p_H: H/H \cap B' \to B'/H \cap B'$.  The assignment $H \mapsto (H \cap B', \pr_m H, \bar p_H)$ defines an isomorphism of $\cO_{\uk,\ur}$ with the variety $\mathcal{V}$ given by
\[
\mathcal{V} = \left\{(H',K,\varphi) \,\Big|\, \begin{array}{c}
\text{$H' \in \cO_{\uk',\ur'}$, $K \in \cC_{k_m,r_m}(B_m)$,} \\
\text{and $\varphi: K \to B'/H'$ a linear map}
\end{array} \right\}.
\]
Since $\mathcal{V}$ is a vector bundle over $\cO_{\uk',\ur'} \times \cC_{k_m,r_m}(B_m)$, we are done.

Next, we prove~\eqref{it:filt-orbit}.  It is clear that the extra factors copies of $\C^\times$ that appear in $GQ_B$ preserve each $\cO_{\uk,\ur}$, so it is enough to show that each $\cO_{\uk,\ur}$ is a single $Q_B$-orbit.  For this claim, we again proceed by induction on $m$.  If $m = 1$, the statement is trivial.  Otherwise, let $B'$ be as above.  Let
\[
U_{B'} = \{ g \in U_B \mid \text{$g(v) = v$ for $v \in B_m$} \},
\qquad
X = \{ g \in U_B \mid \text{$g(v) = v$ for $v \in B'$} \}.
\]
Then $X$ is a normal subgroup of $U_B$, and $U_B = U_{B'} \ltimes X$.  By induction, the variety $\cO_{\uk',\ur} \times \cC_{k_m,r_m}(B_m)$ is a single orbit under the group $I_B \ltimes U_{B'} = Q_{B'} \times \Isoc(B_m)$.

On the other hand, given $H' \in \cO_{\uk',\ur'}$ and $K \in \cC_{k_m,r_m}(B_m)$, the group $X$ acts transitively on the set of linear maps $K \to B'/H'$.  That is, $X$ acts transitively on the fibers of the first map in~\eqref{eqn:afffib}.  The desired conclusion follows.

For part~\eqref{it:filt-p1}, the statement for $Q_B$ follows from Proposition~\ref{prop:orbit-symp}\eqref{it:orbit-stab} and Proposition \ref{prop:orbit-orth}\eqref{it:orbit-orthstab}.  This statement implies the version for $GQ_B$ by Remark~\ref{rmk:giso-defn}.

Similarly, part~\eqref{it:filt-oddvan} follows from Proposition~\ref{prop:orbit-symp}\eqref{it:orbit-eqcoh} and Proposition \ref{prop:orbit-orth}\eqref{it:orbit-ortheqcoh}.
\end{proof}

%--------------------------------------------------------------------------
\subsection{Tangent spaces and slices}
\label{ss:tanslice}
%--------------------------------------------------------------------------

Let $(\uk,\ur) \in \Omega_k$, and let $H \in \cO_{\uk,\ur}$.  For $1 \le i \le m$, choose a complement $K_i$ to $H \cap B_{<i}$ in $H \cap B_{\le i}$.  We thus have
\[
H \cap B_{\le i} = K_1 \oplus \cdots \oplus K_i.
\]
Next, for each $i$, apply Lemma~\ref{lem:v4-decomp} to the subspace $\pr_i H \in \Gr_{k_i}(B_i)$ to obtain a decomposition
\[
B_i = M_{i,1} \oplus M_{i,2} \oplus M_{i,3} \oplus M_{i,4}
\]
with
\[
\pr_i H = M_{i,1} \oplus M_{i,2}.
\]
Via the isomorphism $K_i \simto \pr_i H$, we transfer this to a decomposition
\[
K_i = K_{i,1} \oplus K_{i,2}
\qquad\text{where}\qquad
K_{i,1} \cong M_{i,1}
\quad\text{and}\quad
K_{i,2} \cong M_{i,2}.
\]
We obtain a decomposition
\[
B_{\le i} = \bigoplus_{j = 1}^i (K_{j,1} \oplus K_{j,2} \oplus M_{j,3} \oplus M_{j,4}).
\]
By Lemma~\ref{lem:gr-tan}, we obtain an isomorphism
\[
T_H\Gr_k(V) = \bigoplus_{1\le i, j \le m} \Hom(K_{i,1} \oplus K_{i,2}, M_{j,3} \oplus M_{j,4}).
\]
In the spirit of~\eqref{eqn:gr-tan-decomp}, we can write elements of $T_H\Gr_k(V)$ as $m^2$-tuples of $2 \times 2$ matrices:
\begin{multline}\label{eqn:gr-tan-sum}
T_H\Gr_k(V) = \\ \left\{ (A_{ij})_{1 \le i,j, \le m} \, \Big|\,
A_{ij} = 
\begin{bmatrix}
a_{ij} & b_{ij} \\
c_{ij} & d_{ij}
\end{bmatrix}
\in \Hom(K_{i,1} \oplus K_{i,2}, M_{j,3} \oplus M_{j,4}) \right\}.
\end{multline}

\begin{lem}\label{lem:filt-tan}
Let $H \in \Gr_k(V)$ be as above, and let $\cC$ be the $Q_B$-orbit of $H$.  In terms of~\eqref{eqn:gr-tan-sum}, the tangent space to $\cC$ at $H$ is given by
\[
T_H\cC = \left\{ \left( A_{ij} = 
\begin{bmatrix}
a_{ij} & b_{ij} \\
c_{ij} & d_{ij}
\end{bmatrix} \right)_{1 \le i,j \le m} \,\Big|\,
\begin{array}{c}
\text{$A_{ij} = 0$ if $i < j$, and} \\
\text{$c_{ii}^\dag = -c_{ii}$ for all $i$}
\end{array}
\right\}.
\]
\end{lem}
\begin{proof}
We follow the strategy of the proof of Lemma~\ref{lem:gr-tan-orbit}.  Consider the first the vector space
\[
\fgl(V) = \bigoplus_{1 \le i,j \le m} \Hom(B_i,B_j).
\]
The Lie algebra of $Q_B$ can be written as
\[
\Lie(Q_B) = \bigoplus_{1 \le i \le m} \Lie(\Iso(B_i)) \oplus \bigoplus_{1 \le j < i \le m} \Hom(B_i, B_j).
\]
As in the proof of Lemma~\ref{lem:gr-tan}, there is a natural map $\fgl(V) \to T_H\Gr_k(V)$, leading to the description in~\eqref{eqn:gr-tan-sum}.  The image of $\Lie(Q_B)$ under this map is
\[
\bigoplus_{1 \le i \le m} \left(\begin{array}{c} \text{image of} \\
\Lie(\Iso(B_i)) \end{array} \right)
\oplus
\bigoplus_{1 \le j < i \le m} \Hom(K_{i,1} \oplus K_{i,2}, M_{j,3} \oplus M_{j,4}).
\]
In terms of~\eqref{eqn:gr-tan-sum}, this shows that a point $(A_{ij})_{1 \le i,j \le m}$ in the image satisfies $A_{ij} = 0$ whenever $j > i$.  When $j = i$, $A_{ij}$ must lie in the image of $\Lie(\Iso(B_i)) \to \Gr_{k_i}(B_i)$, and this image is described by Lemma~\ref{lem:gr-tan-orbit}.
\end{proof}

\begin{cor}\label{cor:filt-slice1}
Suppose $H \in \Gr_k(V)$ satisfies
\begin{equation}\label{eqn:H-niceslice}
H = \pr_1 H \oplus \pr_2 H \oplus \cdots \oplus \pr_m H,
\end{equation}
and let $\cC$ be the $Q_B$-orbit of $H$.  There is a transverse slice $S$ to $\cC$ at $H$ and an isomorphism
\begin{equation}\label{eqn:filt-slice}
S \cong 
\left\{ \left( A_{ij} = 
\begin{bmatrix}
a_{ij} & b_{ij} \\
c_{ij} & d_{ij}
\end{bmatrix} \right)_{1 \le i,j \le m} \,\Big|\,
\begin{array}{c}
\text{$A_{ij} = 0$ if $i > j$, and for all $i$,} \\
\text{$A_{ii} = [\begin{smallmatrix} 0 & 0 \\ c_{ii} & 0 \end{smallmatrix}]$ where $c_{ii}^\dag = -c_{ii}$ for all $i$}
\end{array}
\right\}.
\end{equation}
that identifies $H \in S$ with $0 \in \Hom(M_1,M_4)$.

Moreover, there is a cocharacter $\psi: \Gm \to GI_B$ such that the induced $\Gm$-action on $\Gr_k(V)$ preserves $S$, and acts on it linearly with strictly positive weights.
\end{cor}
\begin{proof}
Thanks to the assumption~\eqref{eqn:H-niceslice}, we may choose
\begin{equation}\label{eqn:K-niceslice}
K_i = \pr_i H \subset B_i
\end{equation}
for all $i$.  We work with this choice for the remainder of this proof.

Let $X$ denote the right-hand side of~\eqref{eqn:filt-slice}.  In view of of~\eqref{eqn:gr-tan-sum} and Lemma~\ref{lem:filt-tan}, we can regard $X$ as a linear subspace of $T_H\Gr_k(V)$ that is complementary to $T_H\cC$.  Via Lemma~\ref{lem:gr-tan}\eqref{it:gr-open}, we can identify this subspace with a locally closed subvariety $S \subset \Gr_k(V)$ that meets $\cC$ transversely.

Next, choose a sequence of integers
\[
e_1 < e_2 < \cdots < e_m.
\]
Define $\psi_1: \Gm \to GI_B$ and $\psi_2: \Gm \to I_B$ as follows:
\begin{align*}
\psi_1(z)(v) &= z^{e_i}v\quad\text{if $v \in B_i$,}
&
\psi_2(z)(v) &= 
\begin{cases}
z^{-1}v & \text{if $v \in K_{i,1}$,} \\
zv & \text{if $v \in M_{i,4}$,} \\
v & \text{if $v \in K_{i,2}$ or $v \in M_{1,3}$.}
\end{cases}
\end{align*}
(Here, the assumption~\eqref{eqn:K-niceslice} guarantees that $\psi_2$ actually takes values in $I_B$.)  Then define
\[
\psi: \Gm \to GI_B
\qquad\text{by}\qquad
\psi(z) = \psi_1(z)\psi_2(z).
\]
The induced action on $T_H\Gr_k(V)$ is given by
\[
\psi(z)\left(\begin{bmatrix}
a_{ij} & b_{ij} \\ c_{ij} & d_{ij}
\end{bmatrix}\right)
=
\begin{bmatrix}
z^{e_j-e_i+1} a_{ij} & z^{e_j-e_i}b_{ij} \\ z^{e_j-e_i+2} c_{ij} & z^{e_j-e_i+1}d_{ij}
\end{bmatrix}
\]
In particular, the action on $S$ has strictly positive weights.
\end{proof}

\begin{cor}\label{cor:filt-slice}
Let $H \in \cC$ be as in Lemma~\ref{lem:filt-tan}, and let $\cC$ be the $Q_B$-orbit of $H$.  There is a transverse slice $S$ to $\cC$ at $H$ and a cocharacter $\psi: \Gm \to GQ_B$ such that the induced $\Gm$-action on $\Gr_k(V)$ preserves $S$, and acts on it linearly with strictly positive weights.
\end{cor}
\begin{proof}
Consider the point $H' \in \Gr_k(V)$ given by
\[
H' = \pr_1 H \oplus \pr_2 H \oplus \cdots \oplus \pr_m H.
\]
By construction, $\pr_i H' = \pr_i H$ for all $i$, so by Proposition~\ref{prop:filt-orbits}, they lie in the same $GQ_B$-orbit $\cC$.  In Corollary~\ref{cor:filt-slice1}, we have constructed a transverse slice to $\cC$ at $H'$.  By acting by a suitable element of $GQ_B$, we obtain the desired slice to $\cC$ at $H$.
\end{proof}

%%%%%%%%%%%%%%%%%%%%%%%%%%%%%%%%%%%%%%%%%%%%%%%%%%%%%%%%%%%%%%%%%%%%%%%%%%%
\section{Resolution of singularities for orbit closures}
\label{sec:resoln2}
%%%%%%%%%%%%%%%%%%%%%%%%%%%%%%%%%%%%%%%%%%%%%%%%%%%%%%%%%%%%%%%%%%%%%%%%%%%

Let $(\uk,\ur) \in \Omega$.  The goal of this section is to construct a resolution of singularities $\mu: X(\uk,\ur) \to \overline{\cO_{\uk,\ur}}$, generalizing Proposition~\ref{prop:m1-resoln}.  We will also prove that the cohomology of $X(\uk,\ur)$ and of the fibers of $\mu$ is even and free.

We define $X(\uk,\ur)$ as follows.  A point of $X(\uk,\ur)$ is a list of $3m$ vector spaces, denoted by
\[
P_i, H_i \subset B_{\leq i}
\qquad\text{and}\qquad
\tilde P_i \subset B_i
\qquad\text{for $1 \le i \le m$,}
\]
satisfying the following conditions\footnote{If $r_i$ is one of the special symbols $0'$ or $0''$, it should be interpreted as $0$ in these equations.} for each $i$:
\[
\begin{aligned}
\dim \tilde P_i &= k_i - r_i, \\
\dim P_i &= k_1 + \cdots + k_{i-1} + k_i - r_i, \\
\dim H_i &= k_1 + \cdots + k_{i-1} + k_i, \\
\end{aligned}
\qquad
\begin{aligned}
\tilde P_i &\subset \tilde P_i^\perp, \\
\pr_i P_i&\subset \tilde P_i, \\
\pr_i H_i &\subset \tilde P_i^\perp, \\
H_{i-1} &\subset P_i \subset H_i
\end{aligned}
\]
with one extra condition:
\[
\text{if $r_i = 0'$ or $0''$, then $\tilde P_i \in \cC_{k_i,r_i} \subset \Gr_{k_i}(B_i)$.}
\]

By construction, $X(\uk,\ur)$ is a closed subvariety of a product of Grassmannians, so it is a projective variety.  The defining conditions imply that each $\tilde P_i$ is an isotropic subspace of $B_i$, and that $P_1 = \tilde P_1$.  On the other hand, we have $\dim H_m = k$.  There is a map
\[
\mu: X(\uk,\ur) \to \Gr_k(B)
\]
that sends a point to $H_m$.  

\begin{lem}\label{lem:xkr-smooth}
The variety $X(\uk,\ur)$ is a smooth, connected, and simply-connected.  Moreover, for any commutative ring $\bk$, $\coh^\bullet(X(\uk,\ur),\bk)$ is even and free. 
\end{lem}
\begin{proof}
Let $j$ be an integer with $0 \le j \le m$.  We define varieties $Y_j$ and $Y'_j$ as follows: a point is a collection of vector spaces as indicated below, subject to the same conditions as those in the definition of $X(\uk,\ur)$.
\[
\begin{array}{llcll}
\multicolumn{2}{c}{\text{\textit{For $Y_j$}}} &&
\multicolumn{2}{c}{\text{\textit{For $Y'_j$}}} \\
\cline{1-2}\cline{4-5}
\tilde P_i^{\strut} & \text{for $1 \le i \le m$} &&
\tilde P_i & \text{for $1 \le i \le m$}  \\
P_i & \text{for $1 \le i \le j$} &&
P_i, H_i & \text{for $1 \le i \le j$} \\
H_i & \text{for $1 \le i < j$} 
\end{array}
\]
There is an obvious sequence of maps
\begin{equation}\label{eqn:xsm2}
X(\uk,\ur) = Y'_m \to Y_m \to Y'_{m-1} \to \cdots \to Y'_1 \to Y_1,
\end{equation}
where each map forgets one subspace from the data.  

Let us describe the fibers of the maps in~\eqref{eqn:xsm2}.  Let $1 \le j \le m$, and fix a point in $Y_j$.  The fiber of $Y'_j \to Y_j$ over the given point is the variety of subspaces $H_j$ such that
\[
P_j \subset H_j  \subset B_{<j}\oplus \tilde P_j^\perp 
\qquad\text{and}\qquad
\dim H_j = k_1 + \cdots + k_j.
\]
The space of choices for $H_j$ is isomorphic to the Grassmannian
\[
\textstyle\Gr_{r_j}((B_{<j} \oplus \tilde P_j^\perp)/P_j),
\]
and we conclude that $Y'_j \to Y_j$ is a Grassmannian bundle, i.e., a locally trivial fiber bundle whose fibers are Grassmannians.  Similar arguments show that $Y_{j+1} \to Y'_j$ is also a Grassmannian bundle.

For each variety $Y$ appearing in~\eqref{eqn:xsm2}, we will prove the following claims:
\begin{enumerate}
\item $Y$ is smooth, connected, and simply-connected.
\item We have $\coh^\bullet(Y,\bk)$ is even and free.
\end{enumerate}
We proceed by induction, starting with $Y_1$.  Since $P_1 = \tilde P_1$, and since each $\tilde P_i$ is isotropic, we may identify
\begin{equation}\label{eqn:xsm1}
Y_1 = Y_{11} \times \cdots \times Y_{im}
\quad\text{where}\quad
Y_{1i} =
\begin{cases}
\cC_{k_i,r_i}(B_i) & \text{if $r_i = 0'$ or $0''$,} \\
\Gr_{k_i - r_i}(B_i)^\iso & \text{otherwise.}
\end{cases}
\end{equation}
Recall that when $r_i = 0'$ or $0''$, the variety $\cC_{k_i,r_i}(B)$ is a connected component of $\Gr_{k_i}(B_i)^\iso$.  Using Lemma~\ref{lem:griso-conn}, we see that each $Y_{1i}$ is smooth, connected, and simply-connected, and by Corollary~\ref{cor:gr-oddvan}, $\coh^\bullet(Y_{1i},\bk) = 0$ is even and free.  We deduce that $Y_1$ has the desired properties as well.

For the inductive step, let $Y' \to Y$ be a map in~\eqref{eqn:xsm2}, and assume that the claims are known for $Y$.  We have seen that $Y' \to Y$ is a Grassmannian bundle.  It follows immediately that $Y'$ is smooth, connected, and simply-connected.  Since $Y$ is simply-connected, the Serre spectral sequence shows that $\coh^\bullet(Y') = 0$ is even and free.
\end{proof}

\begin{prop}\label{prop:xkr-image}
The image of $\mu: X(\uk,\ur) \to \Gr_k(B)$ is $\overline{\cO_{\uk,\ur}}$ and the restriction $\mu|_{\mu^{-1}(\cO_{\uk,\ur})}: \mu^{-1}(\cO_{\uk,\ur} )\to \cO_{\uk,\ur}$ is an isomorphism of varieties.  As a consequence,
\[
\mu: X(\uk,\ur) \to \overline{\cO_{\uk,\ur}}
\]
is a resolution of singularities.
\end{prop}
\begin{proof}
For $x \in X(\uk,\ur)$, we have
\begin{align*}
\dim \pr_i P_i(x) &\le k_i - r_i, \\
\dim \pr_i H_i(x) &=\dim H_i/{B_{<i}\cap H_i}\le \dim H_i/P_i  \le k_i.
\end{align*}
Let $U$ be the open subset of $X(\uk,\ur)$ where these inequalities are equalities for all $i$.  The definitions imply that for $x \in U$, we have
\[
\pr_i P_i = \tilde P_i
\qquad
\text{and}\qquad
\tilde P_i \subset \pr_i H_i \subset (\tilde P_i)^\perp.
\]
In other words, $\tilde P_i$ is contained in the radical of $\pr_i H_i$. Therefore the first assertion of the proposition is true.

Since $\mu$ is proper and $Q_B$-equivariant, its image is closed, irreducible, and stable under $Q_B$, and is therefore the closure of a single $Q_B$-orbit.

Let $H\in \cO_{\uk,\ur}$. We can define $H_m=H$. Clearly $\dim \pr_i H\cap \pr_i H^\perp= k_i-r_i$. So we define $\tilde P_i= \pr_i H\cap \pr_i H^\perp$ which has dimension $k_i-r_i $ and an isotropic subspace. Define $H_i=H\cap B_{\leq i}$. Then $\dim H_1=\dim H\cap B_{\leq 1}=\dim \pr_1 H= k_1$. Now we prove $\dim H_i=k_1+\dots +k_i$ by induction. Let the statement is true for $H_{i-1}$. Then $\dim H_i=\dim \pr_i H + \dim H\cap B_{\leq i-1}=\dim  \pr_i H + \dim H_{i-1}=k_i+k_{i-1}+\dots +k_1 $. Also $\pr_i H_i=\pr_i H \subset \pr_i H +\pr_i H ^\perp= \tilde P_i^\perp$. We have the projection map $\pi: H_i\to \pr_i H$, where $\pr_i H$ contains $\tilde P_i$. We define $P_i$ to be the pre-image of $\tilde P_i$ under the map $\pi$. Then $\dim P_i=\dim \tilde P_i +\ker \pi=k_i-r_i+ \dim H\cap B_{<i}=k_i-r_i+k_{i-1}+\dots+k_1$. Also $\ker \pi= H_{i-1}\subset P_i$. The $i$-th projection, $\pr_i P_i \subset \tilde P_i$. So it is clear that the choice of $H\in \cO_{\uk,\ur}$ determines the element in $X_{\uk,\ur}$, hence $\mu|_{\mu^{-1}(\cO_{\uk,\ur})}: \mu^{-1}(\cO_{\uk,\ur} )\to \cO_{\uk,\ur}$ is a bijection.  Since a bijective morphism of smooth complex varieties is an isomorphism, we are done.
\end{proof}

\begin{rmk}\label{rmk:xkr-inverse}
From the preceding proof, we can extract a description of the inverse isomorphism $\cO_{\uk,\ur} \to \mu^{-1}(\cO_{\uk,\ur})$: it sends $H \in \cO_{\uk,\ur}$ to the collection $( \tilde P_i, P_i, H_i : 1 \le i \le m)$ with
\[
\tilde P_i = \rad (\pr_i H), 
\qquad P_i = \text{preimage of $\tilde P_i$ in $H \cap B_{\le i}$,}
\qquad H_i = H \cap B_{\le i}.
\]
\end{rmk}

The remainder of this section is devoted to the study of the fibers of $\mu$.  To do this, we introduce another variety $F(\uk,\ur)$ as follows.  A point of $F(\uk,\ur)$ is a collection of vector spaces 
\[
P_i, H_i \subset B_{\le i}  \qquad\text{for $1 \le i \le m$}
,\]
whose dimensions are as in the definition of $X(\uk,\ur)$, and that satisfy the following containment relations:
\begin{equation}\label{eqn:fkr-defn}
H_{i-1} \subset P_i \subset H_i \subset B_{\le i}
\qquad\text{and}\qquad
\pr_i P_i, \pr_i H_i \subset (\pr_i P_i)^\perp.
\end{equation}

We claim that there is an obvious map
\[
X(\uk,\ur) \to F(\uk,\ur)
\]
given by forgetting the $\tilde P_i$'s.  For this to make sense, we must check that points of $X(\uk,\ur)$ satisfy the additional condition~\eqref{eqn:fkr-defn}.  This condition is easily deduced from the fact that $\pr_i P_i \subset \tilde P_i\subset \tilde P_i^\perp$. This implies $ \pr_i P_i \subset \tilde P_i^\perp \subset (\pr_i P_i)^\perp $.
Also we have  $\pr_i H_i \subset \tilde P_i^\perp \subset (\pr_i P_i)^\perp $.  
There is an obvious map
\[
\mu_F: F(\uk,\ur) \to \Gr_k(B).
\]

\begin{lem}\label{lem:fkr-pave}
Let $M \in \Gr_k(B)$.  The variety $\mu_F^{-1}(M)$ admits an affine paving
\[
\mu_F^{-1}(M) = \bigsqcup_{\alpha \in I} S_\alpha
\]
such that the following functions are constant on each piece $S_\alpha$ of the paving:
\[
\dim P_i \cap B_{<i},\qquad  \dim H_i \cap B_{<i},\qquad \dim P_i \cap (\rad \pr_i M  + B_{<i}).
\]
\end{lem}
In this statement, ``$\rad \pr_i M$'' is the radical of the restriction of the bilinear form $\lag{-},{-}\rag_{B_i}$ to $\pr_i M$.  In other words, $\rad \pr_i M = \pr_i M \cap (\pr_i M)^\perp$.
\begin{proof}
Let $1 \le j \le m$.  Define varieties $F_j$ and $F_j'$ as follows: a point is a collection of subspaces as indicated below, subject to the same conditions as in the definition of $F(\uk,\ur)$, along with the additional conditions indicated below.
\[
\begin{array}{llcll}
\multicolumn{2}{c}{\text{\textit{For $F_j$}}} &&
\multicolumn{2}{c}{\text{\textit{For $F'_j$}}} \\
\cline{1-2}\cline{4-5}
P_i & \text{for $1 \le i \le j$,} &&
P_i, H_i & \text{for $1 \le i \le j$,} \\
H_i & \text{for $1 \le i < j$}
\vspace{5pt} \\
P_j \subset M \cap B_{\le j} &&&
H_j \subset M \cap B_{\le j} \\
\pr_i M \subset (\pr_i P_i)^\perp & \text{for $1 \le i \le j$} &&
\pr_i M \subset (\pr_i P_i)^\perp & \text{for $1 \le i \le j$}
\end{array}
\]
We have a sequence of maps
\[
\mu_F^{-1}(M) = F'_m \to F_m \to F'_{m-1} \to \cdots \to F_2 \to F'_1  \to F_1.
\]
We will show that every variety in this sequence admits an affine paving with the property that the functions mentioned in the statement of the lemma are all constant on each piece of the paving.  The proof is by induction along the sequence of maps above.

Let us prove the claim for $F_j$, assuming that $F'_{j-1}$ has an appropriate affine paving.  (This paragraph also covers the start of the induction: interpret $F'_0$ to mean a singleton, and take $P_0 = H_0 = 0$.)  Consider the bilinear form $({-},{-})_{B_{\le j}}$ on $B_{\le j}$ defined by declaring $B_{<j}$ to be in its radical, and by declaring its restriction to $B_j$ to coincide with $\lag{-},{-}\rag_{B_j}$.  Then the condition $\pr_j P_j$ be isotropic with respect to $\lag{-},{-}\rag_{B_j}$ is equivalent to requiring $P_j$ to be isotropic with respect to $({-},{-})_{B_{\le j}}$.  Fix a stratum $S$ in $F'_{j-1}$.   Each point $s \in S$ includes, as part of its data, a subspace $H_{j-1}(s) \subset B_{<j}$.  The quotient $B_{\le j}/H_{j-1}(s)$ has a bilinear form induced by $({-},{-})_{B_{\le j}}$, and 
\[
F_j \times_{F'_{j-1}} S = \left\{ (s, P_j) \,\Big|\, 
\begin{array}{@{}c@{}}
\text{$s \in S$, $H_{j-1}(s) \subset P_j \subset M \cap B_{\le j}$, and} \\
\text{$P_j$ is isotropic with respect to $({-},{-})_{B_{\le j}}$}
\end{array} \right\}.
\] which is same as,
\[\{(s,\bar{P_j})\mid \bar{P_j}\subset M\cap B_{\leq j}/H_{j-1}(s)\}. \] Now $M\cap B_{<j}/H_{j-1}(s)\subset \rad \pr_j M+B_{<j}/H_{j-1}(s)$ is a sequence of isotropic subspaces in $ M\cap B_{\leq j}/H_{j-1}(s) $.
Therefore Lemma~\ref{lem:gr-pave} implies that $F_j \times_{F''_{j-1}} S \to S$ admits a bundle of affine pavings such that the functions 
\[
\dim P_j \cap B_{<j}, \qquad \dim P_j \cap (\rad \pr_j M + B_{<j})
\]
are constant on strata.  It follows that $F_j$ itself admits an affine paving such that the desired functions are constant on strata.

Next, we prove the claim for $F'_j$, assuming that $F_j$ has an appropriate affine paving.  Fix a stratum $S$ in $F_j$.  Each point $s \in S$ determines a subspace $P_j(s) \subset B_{\le j}$.  We have
\[
F'_j \times_{F_j} S = \{ (s, H_j) \mid 
\text{$s \in S$ and $P_j(s) \subset H_j \subset M \cap B_{\le j}$} \}.
\]
Similar reasoning shows that $F'_j \times_{F_j} S \to S$ admits a bundle of affine pavings such that $\dim H_j \cap B_{<i}$ is constant on strata, so $F'_j$ admits an affine paving with the desired properties as well.
\end{proof}

\begin{lem}\label{lem:xkr-affpave}
Let $M \in \Gr_k(B)$.  The variety $\mu^{-1}(M)$ admits an affine paving
\[
\mu^{-1}(M) = \bigsqcup_{\alpha \in I} S_\alpha
\]
such that the following functions are constant on each piece $S_\alpha$ of the paving:
\[
\dim P_i \cap B_{<i},\qquad  \dim H_i \cap B_{<i},\qquad \dim P_i \cap (\rad \pr_i M  + B_{<i}).
\]
\end{lem}
\begin{proof}
Let $S$ be a piece of the affine paving of $\mu_F^{-1}(M) \subset F(\uk,\ur)$ from Lemma~\ref{lem:fkr-pave}.  It is enough to show that
\[
X(\uk,\ur) \times_{F(\uk,\ur)} S
\]
admits an affine paving.  For $s \in S$, we have a collection of spaces $P_i(s), H_i(s) \subset B$ satisfying~\eqref{eqn:fkr-defn}, and with $H_m(s) = M$.  For $1 \le i \le m$, let $Y_{1i}$ be as in~\eqref{eqn:xsm1}.  We have an identification
\[
X(\uk,\ur) \times_{F(\uk,\ur)} S
=
\left\{ (s, \tilde P_1, \ldots, \tilde P_m) \,\Big|\, 
\begin{array}{@{}c@{}}
\text{$s \in S$, $\tilde P_i \in Y_{1i}$, and} \\
\text{$\pr_i P_i(s) \subset \tilde P_i \subset B_i, 1\leq i\leq m-1$}, \\\text{$\pr_m P_m(s)\subset \tilde P_m\subset (\pr_m M)^\perp$}\end{array}\right\}.
\]
But this is same as,
\[\Gr_{a_1}((\pr_1M)^\perp/\pr_1P_1(s))\times \dots \times \Gr_{a_m}((\pr_mM)^\perp/\pr_mP_m(s)),\] where $a_i$ is $\dim \tilde{P}_i-\dim \pr_i P_i$. The last term is independent of the choice of $s\in S$ as, $\dim \pr_i P_i=\dim P_i-\dim P_i\cap B_{<i}$ and $\dim P_i\cap B_{<i}$ is constant on $S$ by Lemma \ref{lem:fkr-pave}. Also as $X(\uk,\ur) \times_{F(\uk,\ur)} S $ is a product of Grassmannians, it admits an affine paving.
\end{proof}

%%%%%%%%%%%%%%%%%%%%%%%%%%%%%%%%%%%%%%%%%%%%%%%%%%%%%%%%%%%%%%%%%%%%%%%%%%%
\section{Compactified \texorpdfstring{$\pi_1$}{pi1}-trivial fibrations}
\label{sec:locsys2}
%%%%%%%%%%%%%%%%%%%%%%%%%%%%%%%%%%%%%%%%%%%%%%%%%%%%%%%%%%%%%%%%%%%%%%%%%%%

In the previous section, we generalized Proposition~\ref{prop:m1-resoln}.  In this section, we generalize Proposition~\ref{prop:fund-ses}: we will construct a proper map
\[
\hmu: \hX(\uk,\ur) \to \overline{\cO_{\uk,\ur}}
\]
whose domain is smooth, and whose restriction to $\pi^{-1}(\cO_{\uk,\ur})$ is a ``$\pi_1$-injective fibration of index $2$'' as defined in Section~\ref{sec:resoln}.

We define $\hX(\uk,\ur)$ as follows: a point of $\hX(\uk,\ur)$ consists of a point of $X(\uk,\ur)$, together with the following additional data:
\begin{itemize}
\item For each $i$ such that $\lag{-},{-}\rag_{B_i}$ is symmetric, $r_i > \max \{0, 2k_i - n_i\}$, and $r_i$ is even, subspaces
\[
Q_i \subset B
\qquad\text{and}\qquad
\tilde Q_i \subset B_i
\]
satisfying the following conditions:
\begin{align*}
\dim \tilde Q_i &= k_i - r_i/2, & \tilde P_i &\subset \tilde Q_i \subset \tilde Q_i^\perp \\
\dim Q_i &= k_1 + \cdots + k_{i-1} + k_i - r_i/2 & \pr_i Q_i & \subset  \tilde Q_i, \\
&& P_i &\subset Q_i \subset H_i
\end{align*}
\item For each $i$ such that $\lag{-},{-}\rag_{B_i}$ is symmetric, $r_i > \max \{0, 2k_i - n_i\}$, and $r_i$ is odd, subspaces
\[
Q_i \subset B \oplus \C
\qquad\text{and}\qquad
\tilde Q_i \subset B_i \oplus \C
\]
satisfying the following conditions:
\begin{align*}
\dim \tilde Q_i &= \textstyle k_i - \frac{r_i-1}{2}, & \tilde P_i &\subset \tilde Q_i \subset (\tilde Q_i)_+^\perp \\
\dim Q_i &= \textstyle k_1 + \cdots + k_{i-1} + k_i - \frac{r_i-1}{2} & &\pr_iQ_i   \subset  \tilde Q_i, \\
&& P_i &\subset Q_i \subset H_i \oplus \C
\end{align*}
(Here, in a slight abuse of notation, ``$\pr_i Q_i$'' should be understood as $Q_i / (Q_i \cap B_{<i}) \subset B_i \oplus \C$.)
\end{itemize}
There is an obvious map
\[
\hX(\uk,\ur) \to X(\uk,\ur)
\]
given by forgetting the $Q_i$'s and the $\tilde Q_i$'s.  We denote the composition of this map with $\mu: X(\uk,\ur) \to \Gr_k(B)$ by
\[
\hmu: \hX(\uk,\ur) \to \Gr_k(B).
\]

\begin{lem}\label{lem:hxkr-smooth}
The variety $\hX(\uk,\ur)$ is  smooth, connected, and simply-connected.  Moreover, for any commutative ring $\bk$, $\coh^\bullet(\hX(\uk,\ur), \bk)$ is even and free.
\end{lem}
\begin{proof}
In the definition of $\hX(\uk,\ur)$, the subspaces $Q_j$ and $\tilde Q_j$ are defined only for certain $j$.  In this proof, it will be convenient to adopt the convention that
\[
Q_j = P_j,\ \tilde Q_j = \tilde P_j
\qquad\text{if }
\begin{cases}
\text{$\lag{-},{-}\rag_{B_j}$ is skew-symmetric, or} \\
\text{$\lag{-},{-}\rag_{B_j}$ is symmetric and $r_j = \max \{0, 2k_j - n_j\}$.}
\end{cases}
\]

The structure of the proof is very similar to that of Lemma~\ref{lem:xkr-smooth}.  Let $j$ be an integer with $0 \le j \le m$.  We define three varieties $Y_j$, $Y'_j$, and $Y''_j$ as follows: a point is a collection of vector spaces as indicated below, subject to the same conditions as those in the definition of $\hX(\uk,\ur)$.
\[
\begin{array}{llcllcll}
\multicolumn{2}{c}{\text{\textit{For $Y_j$}}} &&
\multicolumn{2}{c}{\text{\textit{For $Y'_j$}}} &&
\multicolumn{2}{c}{\text{\textit{For $Y''_j$}}}\\
\cline{1-2}\cline{4-5}\cline{7-8}
\tilde P_i^{\strut}, \tilde Q_i & \text{for $1 \le i \le m$} && 
  \tilde P_i, \tilde Q_i & \text{for $1 \le i \le m$} &&
  \tilde P_i, \tilde Q_i & \text{for $1 \le i \le m$} \\
P_i & \text{for $1 \le i \le j$} &&
  P_i, Q_i & \text{for $1 \le i \le j$} &&
  P_i, Q_i, H_i & \text{for $1 \le i \le j$} \\
Q_i, H_i & \text{for $1 \le i < j$} &&
  H_i & \text{for $1 \le i < j$}
\end{array}
\]

In place of~\eqref{eqn:xsm2}, we now have a sequence of maps
\begin{equation}\label{eqn:hxsm2}
\hX(\uk,\ur) = Y''_m \to Y'_m \to Y_m \to Y''_{m-1} \to \cdots \to Y'_1 \to Y_1.
\end{equation}
Let $1 \le j \le m$.  The same reasoning as in Lemma~\ref{lem:xkr-smooth} shows that $Y'_j \to Y_j$ and $Y_j \to Y''_{j-1}$ are locally trivial Grassmannian bundles.  This reasoning also applies verbatim to $Y''_j \to Y'_j$ except in the case where $\lag{-},{-}\rag_{B_j}$ is symmetric, $r_j > \max \{0, 2k_j - n_j\}$, and $r_j$ is odd.  In this case, an additional comment is required because the definition involves $H_j \oplus \C$ rather than $H_j$.  Specifically, the fiber of this  map over a given point of $Y'_j$ is the set of subspaces $H_j$ such that
\[
Q_j \subset H_j \oplus \C \subset B_{<j}\oplus\tilde P_j^\perp \oplus \C
\qquad\text{and}\qquad
\dim H_j = k_1 + \cdots + k_j.
\]
Let $\bar Q_j$ be the image of $Q_j$ under the projection map $B \oplus \C \to B$.  The condition above is equivalent to
\[
\bar Q_j \subset H_j \subset B_{<j} \oplus \tilde P_j^\perp
\qquad\text{and}\qquad
\dim H_j = k_1 + \cdots + k_j,
\]
so the fiber over the given point of $Y'_j$ is identified with
\[
\Gr_{\frac{r_i-1}{2}}((B_{<j}  \oplus \tilde P_j^\perp)/\bar Q_j).
\]
The subtlety is that in order for these fibers to assemble into a locally trivial bundle, we must check that the dimension of $\bar Q_j$ is constant on $Y'_j$.  This is true: indeed, the argument in the proof of Proposition~\ref{prop:fund-ses} shows that $\dim \bar Q_j = \dim Q_j$.  We conclude that $Y''_j \to Y'_j$ is also a locally trivial Grassmannian bundle.

Next, from the definition, we have
\[
Y_1 = Y_{11} \times \cdots \times Y_{1m}
\]
where
\[
Y_{1i} =
\begin{cases}
\cB_{k_i,r_i}(B) & \text{if $\lag{-},{-}\rag_{B_i}$ is symmetric and $\max\{0,2k_i-n_i\} < r_i$,} \\
\cC_{r_i}(B_i) & \text{if $r_i = 0'$ or $0''$,} \\
\Gr_{k_i-r_i}(B_i)^\iso & \text{in all other cases.}
\end{cases}
\]
Combining Corollary~\ref{cor:gr-oddvan}, Lemma~\ref{lem:griso-conn}, and Proposition~\ref{prop:fund-ses}, we conclude that $Y_1$ is smooth, connected, and simply connected, and that is cohomology is even and free.  Since each map in~\eqref{eqn:hxsm2} is a locally trivial Grassmannian bundle, the same conclusions hold for $\hX(\uk,\ur)$.
\end{proof}

\begin{prop}\label{prop:hxkr-image}
The image of $\hmu: \hX(\uk,\ur) \to \Gr_k(B)$ is $\overline{\cO_{\uk,\ur}}$, and the restricted map
\[
\hmu^\circ = \hmu|_{\hmu^{-1}(\cO_{\uk,\ur})}: \hmu^{-1}(\cO_{\uk,\ur}) \to \cO_{\uk,\ur}
\]
is a $\pi_1$-injective fibration.  If $\bk$ is a ring satisfying~\eqref{eqn:jmw}, the sheaf
\[
\cH^0(\hmu^\circ_*\bk)
\]
is a $Q_B$-equivariant local system in which every indecomposable locally free $Q_B$-equivariant local system on $\cO_{\uk,\ur}$ appears as a direct summand.
\end{prop}

\begin{proof}
Let $H \in \cO_{\uk,\ur}$, and let $F = \hmu^{-1}(H)$.  A point in $F$ consists of data $(\tilde P_i, P_i, H_i, \tilde Q_i, Q_i)$ satisfying varying conditions.  By Proposition~\ref{prop:xkr-image}, the terms $\tilde P_i$, $P_i$, and $H_i$ are determined by $H$.  To describe $\hmu^{-1}(H)$, we must describe the space of choices for $Q_i$ and $\tilde Q_i$.  

The formulas in Remark~\ref{rmk:xkr-inverse} imply that $P_i \cap B_{<i} = H_i \cap B_{<i}$ for all $i$, and that these spaces have dimension $k_1 + \cdots + k_{i-1}$.  From the definition of $\hX(\uk,\ur)$, we see that $Q_i \cap B_{<i}$ must also equal $H_i \cap B_{<i}$.  This implies that $\dim \pr_i Q_i = \dim \tilde Q_i$, and hence that $\tilde Q_i = \pr_i Q_i$.  Phrased another way, $Q_i$ must be the preimage of $\tilde Q_i$ under the map $H \cap B_i \to \pr_i H$ (or if $r_i$ is odd, under $(H \cap B_i) \oplus C \to \pr_i H \oplus \C$).  Thus, it remains to determine the space of choices for $\tilde Q_i$.  The quotient $\tilde Q_i/\tilde P_i$ is an isotropic subspace of $\pr_i H/\tilde P_i$ (or if $r_i$ is odd, of $(\pr_i H \oplus \C)/\tilde P_i$), which carries a nondegenerate bilinear form.  We conclude that the fiber $F = \hmu^{-1}(H)$ is described by
\begin{multline*}
F \cong \prod_{\substack{1 \le i \le m \\ \text{$\lag{-},{-}\rag_{B_i}$ symmetric} \\ \max\{0,2k_i - n_i\} < r_i \\ \text{$r_i$ even}}}
\Gr_{r_i/2}(\pr_i H/ \rad(\pr_i H))^\iso
\quad
\times \\
\prod_{\substack{1 \le i \le m \\ \text{$\lag{-},{-}\rag_{B_i}$ symmetric} \\ \max\{0,2k_i - n_i\} < r_i \\ \text{$r_i$ odd}}}
\Gr_{(r_i+1)/2}((\pr_i H \oplus \C)/ \rad(\pr_i H))^\iso
\end{multline*}
More generally, $\hmu^{-1}(\cO_{\uk,\ur}) \to \cO_{\uk,\ur}$ is a locally trivial fibration with fibers isomorphic to $F$.  

 Using Lemma~\ref{lem:griso-conn}, we see that $\pi_0(F) \cong (\Z/2\Z)^d$ where
\[
d = |\{ i \mid \text{$\lag{-},{-}\rag_{B_i}$ is symmetric and $\max\{0,2k_i - n_i\} < r_i$} \},
\]
and $\pi_1(F) = 1$.  So we get the following  short exact sequence of homotopy groups:
\[
 1 \to \pi_1(\hmu^{-1}(\cO_{\uk,\ur})) \to \pi_1(\cO_{k,r}(V)) \to  (\mb{Z}/2\mb{Z})^d\to 1.
\]
Proposition~\ref{prop:filt-orbits} also identifies the middle term with $(\Z/2\Z)^d$, so $\hmu^{-1}(\cO_{\uk,\ur})$ must be simply connected.  Then the same reasoning as in Proposition~\ref{prop:fund-ses}\eqref{it:push-forward} shows that $\cH^0(\hmu_*^\circ \bk)$ is the local system corresponding to the regular repreentation $\bk[(\Z/2\Z)^d]$.  Since $2$ is invertible in $\bk$, the regular representation decomposes as a direct sum of representations that are each free of rank~$1$ over $\bk$.  Under the additional assumptions from~\eqref{eqn:jmw}, every indecomposable $\bk[(\Z/2\Z)^d]$-module that is free over $\bk$ occurs as a direct summand of the regular representation.  This proves the second assertion in the proposition.
\end{proof}

Let $\hF(\uk,\ur)$ be the variety defined as follows: a point of $\hF(\uk,\ur)$ consists of a point of $F(\uk,\ur)$ and $\tilde P_i, 1\leq i\leq m$ satisfying the relations in $X(\uk,\ur)$ together with the following additional data: for each $i$ such that $\lag{-},{-}\rag_{B_i}$ is symmetric and $r_i > \max\{0,2k_i-n_i\}$, we require a space $Q_i$ where
\begin{align*}
Q_i & \subset B, & P_i &\subset Q_i \subset H_i, & \pr_i Q_i &\subset (\pr_i Q_i)^\perp &&\text{if $r_i$ is even,} \\
Q_i & \subset B \oplus \C, & P_i &\subset Q_i \subset H_i \oplus \C & \pr_{i,+}Q_i &\subset (\pr_{i,+}Q_i)^\perp_+ &&\text{if $r_i$ is odd,}
\end{align*}
and where the dimension of $Q_i$ is as in the definition of $\hX(\uk,\ur)$, and where for $r_i$ odd, we set
\[
\pr_{i,+}Q_i = Q_i/(Q_i \cap B_{<i}) \subset B_i \oplus \C.
\]
There are obvious maps
\[
\hX(\uk,\ur) \to \hF(\uk,\ur) \to F(\uk,\ur) 
\qquad\text{and}\qquad
\hmu_F: \hF(\uk,\ur) \to \Gr_k(B)
\]
given by forgetting appropriate parts of the data.  

\begin{lem}\label{lem:hfpkr-affpave}
Let $M \in \Gr_k(B)$.  The variety $(\hmu_F)^{-1}(M)$ admits an affine paving
\[
\hmu_F^{-1}(M) = \bigsqcup_{\alpha \in I} S_\alpha
\]
such that the following functions are constant on each piece $S_\alpha$ of the paving:
\[
\dim P_i \cap B_{<i},\quad  \dim H_i \cap B_{<i},\quad \dim P_i \cap (\rad \pr_i M  + B_{<i}).
\]
\end{lem}

\begin{proof}
	Let $S$ be a piece of the affine paving of $\mu^{-1}(M) \subset X(\uk,\ur)$ from Lemma~\ref{lem:xkr-affpave}.  It is enough to show that
\[
\hF(\uk,\ur) \times_{X(\uk,\ur)} S
\]
admits an affine paving.  For $s \in S$, we have subspaces $P_i(s), H_i(s), \tilde P_i(s) \subset B$ satisfying the condition in $X(\uk,\ur)$, and with $H_m(s) = M$.   We have an identification
\[
\hF(\uk,\ur) \times_{X(\uk,\ur)} S
=
\left\{ (s, (Q_i)_{\text{certain }1\leq i\leq m}) \,\Big|\, 
\begin{array}{@{}c@{}}
\text{$s \in S$,  and} \\
\text{$ P_i(s) \subset Q_i \subset H_i$ or  $H_i\oplus \C$}\end{array}\right\}.
\] In more detail, suppose $\lag{-},{-}\rag_{B_i}$ is symmetric and $r_i > \max\{0,2k_i - n_i\}$.  If $r_i$ is even, then the space of choices for $ Q_i$ over $s \in S$ is identified with
\[
\Gr( H_i(s)/( P_i(s))^\iso.
\]
If $r_i$ is odd, then we instead have
\[
\Gr((H_i(s)\oplus \C )/P_i(s))^\iso.
\]
Combining these cases, we conclude that $\hF(\uk,\ur) \times_{X(\uk,\ur)} S \to S$ is a locally trivial fibration whose fibers are products of isotropic Grassmannians.  The lemma follows.

\end{proof}

\begin{lem}\label{lem:hxkr-affpave}
Let $M \in \Gr_k(B)$.  The variety $\hmu^{-1}(M)$ admits an affine paving
\[
\hmu^{-1}(M) = \bigsqcup_{\alpha \in I} S_\alpha
\]
such that the following functions are constant on each piece $S_\alpha$ of the paving:
\begin{gather*}
\dim P_i \cap B_{<i},\quad  \dim H_i \cap B_{<i},\quad \dim P_i \cap (\rad \pr_i M  + B_{<i}),
\end{gather*}
\end{lem}

\begin{proof}
Let $S$ be a piece of the affine paving of $(\hmu_F)^{-1}(M)$ from Lemma~\ref{lem:hfpkr-affpave}.  We then have
\begin{multline*}
\hX(\uk,\ur) \times_{\hF(\uk,\ur)} S
= \\
\left\{ (s, (\tilde Q_i)_{\text{certain $i \in \{1,\ldots,m\}$}}) \,\Big|\, 
\begin{array}{@{}c@{}}
\text{$s \in S$ and }\\
\text{$\tilde P_i(s) + \pr_i Q_i(s) \subset \tilde Q_i \subset (\tilde Q_i)^\perp$ or $(\tilde Q_i)^\perp_+$}
\end{array} \right\}.
\end{multline*}
In more detail, suppose $\lag{-},{-}\rag_{B_i}$ is symmetric and $r_i > \max\{0,2k_i - n_i\}$.  If $r_i$ is even, then the space of choices for $\tilde Q_i$ over $s \in S$ is identified with
\[
\Gr(\tilde P_i(s)^\perp/(\tilde P_i(s) + \pr_iQ_i(s)))^\iso.
\]
If $r_i$ is odd, then we instead have
\[
\Gr((\tilde P_i(s)^\perp \oplus \C)/(\tilde P_i(s) + \pr_iQ_i(s)))^\iso.
\]
Combining these cases, we conclude that $\hX(\uk,\ur) \times_{\hF(\uk,\ur)} S \to S$ is a locally trivial fibration whose fibers are products of isotropic Grassmannians.  The lemma follows.
\end{proof}

%%%%%%%%%%%%%%%%%%%%%%%%%%%%%%%%%%%%%%%%%%%%%%%%%%%%%%%%%%%%%%%%%%%%%%%%%%%
\section{Main results on parity sheaves}
\label{sec:parity}
%%%%%%%%%%%%%%%%%%%%%%%%%%%%%%%%%%%%%%%%%%%%%%%%%%%%%%%%%%%%%%%%%%%%%%%%%%%

Let $B = B_1 \oplus \cdots \oplus B_m$ be as in Section~\ref{sec:dsums}.  Let $\bk$ be a ring satisfying assumption~\eqref{eqn:jmw}.  In this section, we study the $Q_B$-equivariant derived category of constructible sheaves on $\Gr_k(B)$, denoted by
\[
\Db_{Q_B}(\Gr_k(B),\bk).
\]

%We may also consider equivariance with respect to the smaller group $I_B$.
%
%\begin{lem}
%Let $\bk$ be a ring satisfying~\eqref{eqn:jmw}.  The category $\Db_{Q_B}(\Gr_k(B),\bk)$ is equivalent to the full subcategory of $\Db_{I_B}(\Gr_k(B),\bk)$ consisting of objects that are constructible with respect to the stratification $\{ \cO_{\uk,\ur} \}_{(\uk,\ur) \in \Omega_k}$.
%\end{lem}
%\begin{proof}
%Let $D$ denote the subcategory of $\Db_{I_B}(\Gr_k(B),\bk)$ consisting of objects that are constructible with respect to the stratification $\{ \cO_{\uk,\ur} \}_{(\uk,\ur) \in \Omega_k}$.  Since $U_B$ is a unipotent group, the forgetful functor $\Db_{Q_B}(\Gr_k(B),\bk) \to \Db_{I_B}(\Gr_k(B),\bk)$ is fully faithful.  Certainly this functor takes values in the subcategory $D$.  Conversely, let $j_{\uk,\ur}: \cO_{\uk,\ur} \to \Gr_k(B)$ be the inclusion map of an orbit, and let $\cL$ be a local system of finite type on $\cO_{\uk,\ur}$.  Proposition~\ref{prop:filt-orbits} implies that $\cL$ is $I_B$- and $Q_B$-equivariant, and hence so is $j_{\uk,\ur*}\cL$.  Since objects of the form $j_{\uk,\ur*}\cL$ generate $D$, we conclude that $\Db_{Q_B}(\Gr_k(B),\bk) \cong D$.
%\end{proof}

Consider the collection of all indecomposable $Q_B$-equivariant local systems on $Q_B$-orbits in $\Gr_k(B)$.  The ``even and free'' property from Proposition~\ref{prop:filt-orbits}\eqref{it:filt-oddvan} means that this collection of local systems satisfies the basic assumptions in~\cite[Eq.~(2.1) and~(2.2)]{jmw:ps} needed for the theory of parity sheaves.  In particular, it makes sense to speak of ``$Q_B$-equivariant parity sheaves'' on $\Gr_k(B)$.  By~\cite[Theorem~2.12]{jmw:ps}, we have the following classification result:

\begin{thm}[\cite{jmw:ps}]\label{jmw}
\phantomsection
\begin{enumerate}
\item For each orbit $\cO_{\uk,\ur} \subset \Gr_k(B)$ and each indecomposable locally free $\bk$-local system $\cL$ on $\cO_{\uk,\ur}$, there is (up to isomorphism) at most one indecomposable parity sheaf
\[
\cE(\cO_{\uk,\ur}, \cL)
\]
that is supported on $\overline{\cO_{\uk,\ur}}$ and that satisfies
\[
\cE(\cO_{\uk,\ur},\cL)|_{\cO_{\uk,\ur}} \cong \cL[\dim \cO_{\uk,\ur}].
\]
\item Every parity complex in $\Db_{Q_B}(\Gr_k(B),\bk)$ is isomorphic to a direct sum of objects of the form $\cE(\cO_{\uk,\ur},\cL)[m]$.
\end{enumerate}
\end{thm}

Note that the first part of this result says ``at most one'': the general theory from~\cite{jmw:ps} does not guarantee the existence of $\cE(\cO_{\uk,\ur},\cL)$.

\begin{thm}
For each orbit $\cO_{\uk,\ur} \subset \Gr_k(B)$ and each $Q_B$-equivariant local system $\cL$ on $\cO_{\uk,\ur}$, the parity sheaf $\cE(\cO_{\uk,\ur},\cL)$ exists.  Moreover,
\[
\coh^\bullet(\Gr_k(\cB), \cE(\cO_{\uk,\ur},\cL)[-\dim \cO_{\uk,\ur}])
\]
is even and free.
\end{thm}
\begin{proof}
Consider the map $\hmu: \hX(\uk,\ur) \to \Gr_k(\cB)$.  Let $d = \dim \hX(\uk,\ur)$, and let $\cF = \hmu_*\bk[d]$.  Because $\hX(\uk,\ur)$ is smooth and $\hmu$ is proper, we see that $\cF$ is isomorphic to its own Verdier dual.  The stalk of $\cF$ at a point $M \in \Gr_k(B)$ is given by
\[
\cF_M = \coh^{\bullet + d}(\hmu^{-1}(M),\bk).
\]
Since $\hmu^{-}(M)$ has an affine paving, we see that the stalks of $\cF$ vanish in degrees whose parity does not match that of $d$.  Since the same condition holds for $\D\cF$, we conclude that $\cF$ is a parity complex.

By Theorem~\ref{jmw}, $\cF$ is a direct sum of objects of the form $\cE(\cO_{\uk',\ur'},\cL')[m]$.  Since $\cF$ is supported on $\overline{\cO_{\uk,\ur}}$, any summand $\cE(\cO_{\uk',\ur'},\cL')[m]$ that occurs in $\cF$ must satisfy $\cO_{\uk',\ur'} \subset \overline{\cO_{\uk,\ur}}$.  As a consequence, if $\cF|_{\cO_{\uk,\ur}}$ is nonzero, and if $\cL$ is an irreducible summand of some cohomology sheaf $\cH^i( \cF|_{\cO_{\uk,\ur}} )$, then $\cE(\cO_{\uk,\ur},\cL)$ exists and occurs (up to shift) as a direct summand of $\cF$.

By Proposition~\ref{prop:hxkr-image}, \emph{every} indecomposable locally free $Q_B$-equivariant local system $\cL$ on $\cO_{\uk,\ur}$ occurs as a direct summand of $\cH^{-d}(\cF)$.  We conclude that every $\cE(\cO_{\uk,\ur},\cL)$ exists, and that
\[
\cE(\cO_{\uk,\ur},\cL)[d - \dim \cO_{\uk,\ur}]
\]
occurs as a direct summand of $\cF$.

Finally, we see that $\coh^i(\Gr_k(B), \cE(\cO_{\uk,\ur},\cL)[-\dim \cO_{\uk,\ur})$ is a direct summand of
\[
\coh^\bullet(\Gr_k(B),\cF[-d])
\cong \coh^\bullet(\hX(\uk,\ur), \bk),
\]
and this is even and free by Lemma~\ref{lem:hxkr-smooth}.
\end{proof}

\begin{cor}
If $\cF \in \Db_{Q_B}(\Gr_k(B),\bk)$ is an even complex, then $\coh^\bullet(\Gr_k(B),\cF)$ is even and free.
\end{cor}

\end{document}